\documentclass[11pt]{article}

\newcommand{\floor}[1]{\lfloor #1 \rfloor}

\usepackage{xcolor}
\usepackage{amsfonts}
\usepackage{amsthm}
\usepackage{amsmath}

\newtheorem{lemma}{Lemma}
\newtheorem{theorem}{Theorem}

\theoremstyle{definition}

\usepackage{graphicx}
\usepackage[numbers]{natbib}
\usepackage{algpseudocode}
\usepackage{algorithm2e}
\usepackage{multirow}
\usepackage{hyperref}
\RestyleAlgo{ruled}

\begin{document}

% ========== Edit your name here
\title{On tail inference in scale-free inhomogeneous random graphs}
\author{Daniel Cirkovic$^1$, Tiandong Wang$^2$, Daren B.H. Cline$^1$}
\date{%
    $^1$Department of Statistics, Texas A\&M University, 77843, TX, U.S.A. \\%
    $^2$Shanghai Center for Mathematical Sciences, Fudan University, 200438, Shanghai, China\\[2ex]%
    \today \\[2ex]
}
\maketitle

\begin{abstract}
Both empirical and theoretical investigations of scale-free network models have found that large degrees in a network exert an outsized impact on its structure. However, the tools used to infer the tail behavior of degree distributions in scale-free networks often lack a strong theoretical foundation. In this paper, we introduce a new framework for analyzing the asymptotic distribution of estimators for degree tail indices in scale-free inhomogeneous random graphs. Our framework leverages the relationship between the large weights and large degrees of Norros-Reittu and Chung-Lu random graphs. In particular, we determine a rate for the number of nodes $k(n) \rightarrow \infty$ such that for all $i = 1, \dots, k(n)$, the node with the $i$-th largest weight will have the $i$-th largest degree with high probability. Such alignment of upper-order statistics is then employed to establish the asymptotic normality of three different tail index estimators based on the upper degrees. These results suggest potential applications of the framework to threshold selection and goodness-of-fit testing in scale-free networks, issues that have long challenged the network science community.
\end{abstract}

\medskip

\section{Introduction}

Early research in network science found that many real-world networks exhibit power-law degree distributions \cite{aiello2000random, albert1999diameter, faloutsos1999power, price1965networks}. Loosely speaking, this means that the proportion of nodes with degree exceeding $j$ is approximately proportional to $j^{-\alpha}$ for some $\alpha > 0$ and $j$ sufficiently large. Here, $\alpha$ is called the tail index and networks satisfying this property are commonly referred to as ``scale-free". The prominence of scale-free networks led to the development of models that could imitate such behavior, which in turn allowed researchers to study the recurring consequences that the scale-free property had on a variety of network characteristics. Prominent examples of network models that emerged during this time are the classes of preferential attachment and inhomogeneous random graphs (IRGs) (see for instance \cite{barabasi1999emergence, chung2002average}).

Despite the interest in the scale-free property and, by extension, the large degrees of random graphs, statistical inference procedures focused on the extremal behavior of the degree distribution are underdeveloped. Often, large degrees in a network are the most impactful on its structure. Empirical investigations have found that the connectivity of scale-free networks is vulnerable to the deletion of targeted high degree vertices \cite{albert2000error}. Theoretical studies of popular scale-free network models also confirm the outsized influence large degree nodes have on connectivity \cite{bhamidi2021multiscale}. Hence, in order to make reliable inference on a wide variety of network properties, network scientists require provably accurate statistical procedures that describe the tail behavior of degree distributions.  

When examining the large degrees of scale-free networks, the foremost object of interest is the power-law tail index $\alpha$. Yet, only recently has the Hill estimator \cite{hill1975simple}, a popular estimator of the tail index, been shown to be consistent in the preferential attachment and inhomogeneous random graph models \cite{bhattacharjee2022large, cirkovic2024emergence, wang2019consistency}. To the best of our knowledge, there are no results confirming asymptotic normality of this or any other tail index estimator in scale-free random networks. Furthermore, thresholding methods that estimate the value beyond which a power-law tail fits best, such as the minimum distance method, have no theoretical basis for application in scale-free networks \cite{clauset2009power, drees2020minimum}. Such lack of tools for assessing the goodness of fit has even provoked controversies regarding the applicability of the scale-free model in network science \cite{broido2019scale, holme2019rare, voitalov2019scale}. 

In this paper, we take a vital first step towards providing valid statistical inference for the extreme-value behavior of large degrees in scale-free random networks. In particular, we provide a framework for proving the asymptotic normality of tail index estimators in scale-free IRGs. Inhomogeneous random graphs assign to each node an independently and identically distributed weight that models its degree. It is well-known that the node degrees concentrate strongly around their weights when the weights are large. This dynamic causes the degree distribution to adopt the tail behavior of the weight distribution \cite{bhattacharjee2022large}. Our framework leverages this relationship between the weights and degrees, but interrogates it further. In particular, for the Norros-Reittu and Chung-Lu random graph models, we determine a rate for the number of nodes $k(n) \rightarrow \infty$ such that, with probability tending towards 1, the node with the $i$-th largest weight will have the $i$-th largest degree for all $i \in \{1, \dots, k(n) \}$. We refer to this phenomenon as the \emph{alignment of upper order statistics}. Under the alignment of upper order statistics, we prove that when applied to the degrees of inhomogeneous random graphs, three different tail index estimators can be approximated by their weight-based counterparts. This approximation gives rise to the asymptotic normality of tail index estimators when applied to IRGs, the first such result for scale-free networks. 
The alignment of upper order statistics thus provides a potential avenue for future work that verifies the validity of multiple extreme-value procedures employing the upper degrees of inhomogeneous random graphs.   

The rest of the paper is organized as follows. In Section~\ref{sec:IRG}, we start by discussing important concepts in IRGs, and give conditions under which our asymptotic results will be derived. We then summarize the key finding, i.e. the alignment of upper order statistics, in Section~\ref{sec:orderstat}, and apply it to three different tail index estimators in Section~\ref{sec:application}. Simulation results and concluding remarks are given in Sections~\ref{sec:sim} and \ref{sec:conclusion}, respectively, and Section~\ref{sec:proof} collects all technical proofs.

\subsection{Notation}

For any $n \in \mathbb{N}$, we employ the notation $[n]$ to refer to the set $\{1, \dots, n\}$. If $X_1, X_2, \dots, X_n$ are a collection of random variables, we let $X_{[n]} = \{X_1, X_2, \dots, X_n\}$. For real numbers $x$ and $y$, we let $x \wedge y$ and $x \vee y$ denote their minimum and maximum, respectively.  

\section{Scale-free inhomogeneous random network models}\label{sec:IRG}

\subsection{Ordered inhomogeneous random network}

Inhomogeneous random graph models assign each node a latent weight that confers some notion of connectivity or popularity relative to other nodes in the network. A larger weight is then typically associated with a greater number of links. In order to explicitly analyze the extent to which a large weight confers a large degree, we consider \emph{ordered} inhomogeneous random network models. Instead of assigning identically distributed weights to each node independently, the ordered inhomogeneous network matches the label of each node with the rank of its allocated weight. In said scheme, the largest weight is always assigned to node one, the second largest to node two and so on. Such an ordering is often imposed when the weights are deterministic (see for example \cite{bhamidi2012novel, van2017random}). Assuming that the weights are drawn from a continuous distribution, then an inhomogeneous random network may be recovered from an ordered version by selecting a permutation of the node labels uniformly at random. As such, the typical degree distribution from an ordered inhomogeneous network is the same as its unordered counterpart. 

We now formally introduce the two ordered inhomogeneous random network models under consideration. Suppose $W_{[n]}$ are independent and identically distributed weights from a positively supported continuous distribution function $F$ and let $W_{(1)}(n) \geq W_{(2)}(n) \geq \dots \geq W_{(n)}(n)$ denote the associated decreasing order statistics. Further, let $L(n) = \sum_{i = 1}^n W_{(i)}(n) = \sum_{i = 1}^n W_i$. Both random networks model the symmetric adjacency matrix $A(n) = \{A_{ij}(n)\}_{i, j = 1}^n$ by conditioning on the latent weights $\{W_{(i)}(n)\}_{i \in [n]}$. The first model under consideration is the Norros-Reittu multigraph which assumes that 
\begin{align}
\label{eq:NR}
A_{ij}(n) \mid \{W_{(i)}(n)\}_{i \in [n]} \overset{\text{ind}}{\sim} \text{Poisson}\left( \frac{W_{(i)}(n)W_{(j)}(n)}{L(n)} \right), \qquad \text{for } 1 \leq i \leq j \leq n.
\end{align}
Let $D_i(n) = \sum_{j = 1}^n A_{ij}(n)$ denote the degree of node $i$ for $i \in [n]$. From \eqref{eq:NR}, we have that
\begin{align}
\label{eq:d}
D_i(n) \mid  \{W_{(i)}(n)\}_{i \in [n]}  \sim \text{Poisson}\left( W_{(i)}(n)\right), \qquad \text{for } i \in [n],
\end{align}
since $D_i(n)$ is the sum of conditionally independent Poisson random variables. We secondly consider the Chung-Lu model which assumes that
\begin{align}
\label{eq:CL}
A_{ij}(n) \mid \{W_{(i)}(n)\}_{i \in [n]} \overset{\text{ind}}{\sim} \text{Bernoulli}\left( \frac{W_{(i)}(n)W_{(j)}(n)}{L(n)} \wedge 1 \right), \qquad \text{for } 1 \leq i \leq j \leq n.
\end{align}
Under \eqref{eq:CL}, note that
\begin{align*}
\mathbb{E}\left[D_i(n) | \{W_{(i)}(n)\}_{i \in [n]} \right] = \sum_{j = 1}^n \left( \frac{W_{(i)}(n)W_{(j)}(n)}{L(n)} \wedge 1 \right), \qquad i \in [n].
\end{align*}
As mentioned earlier, despite working in an ordered model, inhomogeneous random graphs may be recovered by taking a random permutation of the node labels. 
\begin{lemma}
\label{lem:perm}
Suppose $\pi$ is a permutation drawn uniformly at random from $S_n$, the set of all permutations on $[n]$. Define for every $i, j \in [n]$ such that $i \leq j$, $\tilde{A}_{ij}(n) = A_{\pi(i)\pi(j)}(n)$. Then $\{\tilde{A}_{ij}(n) \}_{1 \leq i \leq j \leq n}$ is distributed according to an inhomogeneous random graph model.
\end{lemma}
As such, the ordered and unordered models share the same typical degree distribution. Lemma \ref{lem:perm} is proved in Section \ref{sec:perm}. We next investigate the scale-free nature of inhomogeneous random networks.

\subsection{Scale-free property and order statistics}\label{sec:scalefree}

For random networks, the scale-free property is most generally defined through regular variation of the \textit{asymptotic} degree distribution \cite{holme2019rare, naulet2021bootstrap, voitalov2019scale}. That is, we say a random network is scale-free if for every $k \in \mathbb{Z}_+$
\begin{align}
\label{eq:asymp_deg}
\frac{1}{n}\sum_{i = 1}^n \mathbb{P}\left(D_i(n) = k \right) \rightarrow \ \mathbb{P}\left(\mathcal{D} = k \right), \qquad \text{as } n \rightarrow \infty, 
\end{align}
for some random variable $\mathcal{D}$ with regularly varying tail. That is,
\begin{align}
\label{eq:rv_deg}
\mathbb{P}\left(\mathcal{D} > k \right) = l(k) k^{-\alpha}, \qquad k \in \mathbb{Z}_+,
\end{align}
for some slowly varying function $l$ and $\alpha > 0$. It is well-known that, for a wide class of inhomogeneous random graphs, regular variation of the weight distribution implies regular variation of the asymptotic degree distribution \cite{bhattacharjee2022large, cirkovic2024emergence}. In particular, for the Norros-Reittu random network, regular variation of the weight distribution implies regular variation of the typical degree distribution for all $\alpha >0$ \cite{bhattacharjee2022large}. For the Chung-Lu model, regular variation of the weight distribution implies regular variation of the asymptotic degree distribution when $\alpha >1$ \cite{bollobas2007phase}. In both settings, the asymptotic degree distribution has a mixed-Poisson form where
\begin{align*}
\mathbb{P}\left(\mathcal{D} = k \right) = \int_0^\infty \frac{e^{-w}w^k}{k!} \mathbb{P}\left(W_1 \in dw \right), \qquad k \in \mathbb{Z}_+.  
\end{align*}
Regular variation of the asymptotic degree distribution emerges from the weights due to the strong concentration of the Poisson distribution for large rates. Obtaining rates at which the upper degrees concentrate around large weights will later play a large role in quantifying how the regular variation property translates from the weights to the degrees. 

We henceforth assume that the distribution of $W_1$ is regularly varying with tail index $\alpha > 0$. For the Chung-Lu random graph, we will restrict the tail index to $\alpha > 2$. When working with the order statistics $W_{(1)}(n), \dots, W_{(n)}(n)$, regular variation is often more conveniently expressed in terms of the quantile function. Let $U(t) = F^{\leftarrow}(1 - 1/t)$ for $t \geq 1$ where $F^\leftarrow(y) = \inf\{x : F(x) \geq y \}$. In addition, let $\gamma = 1/\alpha$. According to \cite{resnick2007heavy}, Remark 3.3, regular variation of $W_1$ with tail index $\alpha > 0$ is then equivalent to the statement
\begin{align}
\label{eq:rv}
\lim_{t \rightarrow \infty}\frac{U(tx)}{U(t)} = x^{\gamma}, \qquad x > 0.
\end{align}
Hence, we may write $U(t) = x^{\gamma}\ell(t)$ for some slowly varying function $\ell$. Further, using R\'enyi's representation for exponential order statistics, we have the following distributional characterization for the order statistics:
\begin{align}
\label{eq:orddist}
\left(W_{(1)}(n), W_{2}(n), \dots, W_{(n)}(n)\right) \overset{d}{=} \left( U\left(e^{\sum_{j = 1}^n \frac{E_j}{j}}\right), U\left(e^{\sum_{j = 2}^n \frac{E_j}{j}}\right), \dots, U\left(e^{\frac{E_n}{n}}\right) \right),
\end{align}
where $E_1, E_2, \dots, E_n$ are independent unit rate exponential random variables \cite[Section 4.4]{beirlant2006statistics}. 

In order to achieve asymptotic normality of most tail index estimators based on the upper order statistics, it is common to impose a second order condition on $U$ \cite{de1996generalized}. The second order condition assumes that there exists a $\rho \leq 0$ and a function $A$ with constant sign such that for $x > 0$
\begin{align}
\label{eq:2rv}
\lim_{t \rightarrow \infty} \frac{\frac{U(tx)}{U(t)} - x^{\gamma}}{A(t)} = x^\gamma \frac{x^\rho - 1}{\rho}. 
\end{align}
When $\rho = 0$, the limiting function in \eqref{eq:2rv} is understood to be $x^\gamma \log x$. The second order condition provides a rate of convergence for \eqref{eq:rv}. In addition, when $\rho < 0$, the first order relationship is simplified to $U(t) \sim c t^{\gamma}$ where $c = \lim_{t \rightarrow \infty} t^{-\gamma}U(t)$ \cite{haan2006extreme}. The second order condition may be equivalently restated in terms of the survival function \cite[Theorem 2.3.9]{haan2006extreme}, and \eqref{eq:2rv} can be seen to hold for the Fr\'echet distribution with $\rho = -1$, the t-distribution with $\rho = -2$, and the Burr distribution with $1 - F(x) = (1 + x^{-\rho/\gamma})^{1/\rho}$, $\rho < 0$, among other distributions.

We further assume that there exists a $t_0 > 0$ and $\eta \in [0, 1)$ such that for all $t \geq t_0$ and $x \geq 1$
\begin{align}
\label{eq:Ulower}
\frac{U(tx)}{U(t)} \geq (1 - \eta)x^{\gamma} + \eta.
\end{align}
Hence taking $\eta = 0$ in \eqref{eq:Ulower} is equivalent to stating that $\ell$ is eventually non-decreasing. For $\eta \in (0, 1)$, \eqref{eq:Ulower} allows $\ell$ to be  non-increasing, but ensures that $\lim_{x \rightarrow \infty} \ell(x) \geq (1 - \eta)\ell(t_0)$. Thus, in conjunction with the second order condition with $\rho < 0$, \eqref{eq:Ulower} is not very restrictive. In terms of the tails of $F$, \eqref{eq:Ulower} implies that for all $z \geq U(t_0)$ and $x \geq 1$
\begin{align}
\label{eq:Tlower}
\frac{1 - F(zx)}{1 - F(z)} \geq \left(\frac{x - \eta}{1 - \eta}\right)^{-\alpha}.
\end{align} 
For $\eta = 0$, the condition \eqref{eq:Ulower} naturally holds for the Pareto distribution and it can analytically be shown to hold for the half-Cauchy distribution, the Fr\'echet distribution and the Burr distribution. Reminiscent of the distributions considered in \cite{hall1984best, hall1985adaptive}, the tail function $1 - F(x) = \frac{2+x}{x^2}$ for $x \geq 2$ satisfies \eqref{eq:Tlower} for, say, $\eta = 1/4$. Condition \eqref{eq:Ulower} will lead to a lower bound for the spacings of consecutive order statistics, an important ingredient in our proof strategy.

\section{Alignment of upper order statistics}\label{sec:orderstat}

For inhomogeneous random networks, it is well understood that a large node weight typically confers a large degree. As mentioned in Section \ref{sec:scalefree}, this is due to the strong concentration of the degrees around the weights, conditional on the weights being large. Hence, we expect the largest weights to play a fundamental role in determining the upper degrees. In this section, we find that the large weights not only determine the upper degrees, but that the ordering of the large weights governs the ordering of the upper degrees. Let $R_i(n) = \sum_{j = 1}^n 1_{\{ D_i(n) \leq D_j(n) \}}$ denote the number of nodes with degree at least as large as the degree of node $i$. For scale-free inhomogeneous networks we establish a growth rate on the number of nodes, $k(n) \rightarrow \infty$, such that with probability tending towards $1$ as $n \rightarrow \infty$, we have
\begin{align}
\label{eq:ord_aling}
\left(R_1(n), R_2(n), \dots,R_{k(n)}(n)\right) = \left(1, 2, \dots, k(n) \right).
\end{align} 
Our proof also shows that, with probability tending towards $1$, no ties occur among the upper $k(n)$ degrees so that we indeed find the that the ordering of the upper $k(n)$ weights is preserved by the ordering of the upper $k(n)$ degrees. We refer to the statement \eqref{eq:ord_aling} as the alignment of upper order statistics. The requirement on $k(n)$ is presented in Theorem \ref{thm:order}.

\begin{theorem}
\label{thm:order}
Assume that $U$ satisfies \eqref{eq:rv} and \eqref{eq:Ulower}.
Suppose that the network is generated under either a Norros-Reittu model with $\alpha > 0$ or Chung-Lu model with $\alpha > 2$. Further suppose $k(n)$ is such that $k(n) \rightarrow \infty$ and $k(n) \log^{\alpha}(n) /n^{\frac{1}{4\alpha + 1}} \rightarrow 0$, as $n \rightarrow \infty$. Then as $n \rightarrow \infty$,
\begin{align*}
\mathbb{P}\left(\left(R_1(n), R_2(n), \dots,R_{k(n)}(n)\right) = \left(1, 2, \dots, k(n) \right) \right) \rightarrow 1.
\end{align*}
\end{theorem}
Theorem \ref{thm:order} is proved in Section \ref{sec:ord_proof}. It is pursued in three steps. First, we establish that the spacings between the upper $k(n)$ order statistics of regularly varying random variables are sufficiently large. This is done by establishing high probability upper and lower bounds on the upper order statistics of regularly varying random variables. Then we show that the first $k(n)$ degrees concentrate around their weights at a sufficiently fast rate. In fact, this rate determines how far apart the upper order statistics need to be, which in turn determines how fast $k(n)$ is allowed to grow. The main technical ingredients here are Bennett's inequality for the Norros-Reittu model and Chernoff's inequality for the Chung-Lu random graph. Assuming that the first and second steps are achieved, one may determine that
\begin{align*}
D_1(n) > D_2(n) > \dots > D_{k(n)}(n),
\end{align*}
with probability tending towards one as $n \rightarrow \infty$. Hence, the final step required to prove the alignment of the upper order statistics is to show that, with probability tending towards one, none of the degrees $\{D_i(n)\}_{i \in [n]\setminus [k(n)]}$ exceed $D_{k(n)}(n)$. This sequence of steps can be more precisely outlined by showing that each of the following sequences of events,
\begin{equation}
\begin{split}
S(n) =& \left\lbrace \forall i \in [k(n)],  W_{(i)}(n) - W_{(i + 1)}(n) >  2\sqrt{5\log (n) W_{(i)}(n)} \right\rbrace, \\
C(n) =& \left\lbrace \forall i \in [k(n)],  \left| D_i(n) - W_{(i)}(n) \right| < \sqrt{5\log (n) W_{(i)}(n)}  \right\rbrace,  \\
M(n) =& \left\lbrace D_{k(n)}(n) >  \max_{i \in [n] \setminus [k(n)]} D_i(n)  \right\rbrace.
\end{split}
\end{equation}
occur with probability tending towards one as $n \rightarrow \infty$. That is, $S(n)$, $C(n)$ and $M(n)$ refer to the events described in the first, second and final step, respectively. Theory regarding the events $S(n)$ and $C(n)$ can be found in Sections \ref{sec:space_proof} and \ref{sec:conc_proof}, respectively. Statements for $M(n)$ are developed in Section \ref{sec:ord_proof}.

\section{Application to tail index estimators}\label{sec:application}

In the statistical inference of scale-free networks, a fundamental task is the estimation of the tail index $\alpha$. However, for network data in particular, this is a difficult problem. Theoretical justification for procedures used to estimate the tail index are difficult to obtain due to the degree dependence present in network data. In fact, only consistency of the well-known Hill estimator for the power-law tail index has been verified in two non-trivial network models: the preferential attachment model \cite{wang2019consistency} and the inhomogeneous random graph \cite{bhattacharjee2022large}. To the best of our knowledge, the current literature contains no results regarding the asymptotic normality of \textit{any} tail index estimator for network data.  Further, a variety of practical difficulties also emerge when applying tail index estimators to real-world network data \cite{voitalov2019scale}. For example, ties in integer-valued data often lead to the erratic behavior of tail index estimators \cite{matsui2013estimation}. 

In this section, we find that the alignment of upper order statistics also provides a path to proving the asymptotic normality of tail index estimators when applied to degrees from the Norros-Reittu and Chung-Lu random graphs. Although we focus on three estimators, the presented techniques can in principle be applied to most tail index estimators based on the largest degrees of an inhomogeneous random graph. The three estimators under consideration are the Hill \cite{hill1975simple}, Pickands \cite{pickands1975statistical} and probability weighted moment (PWM) estimators \cite{hosking1987parameter}. All three estimators target the reciprocal of the tail index, $\gamma \equiv 1/\alpha$. In the iid setting, asymptotic normality has been verified for the Hill and Pickands estimators when $\alpha > 0$ and for the PWM estimator when $\alpha > 2$. In all three cases, asymptotic normality relies on the second-order condition for $U$ presented in Section \ref{sec:scalefree}; see Chapter 3 of \cite{haan2006extreme} for detailed discussions. 

Our strategy to prove asymptotic normality of these estimators when applied to the network degrees is to approximate them by their weight-based counterparts. We denote a generic tail index estimator based on the degrees by $\hat{\gamma}_D(n)$ and its weight-based counterpart by $\hat{\gamma}_W(n)$. In all three cases, the upper-order statistic alignment allows us to write $\hat{\gamma}_D(n)$ as a function of the degrees $D_{1}(n), D_2(n), \dots$ rather than the order statistics. From there, an approximation by $\hat{\gamma}_W(n)$ is obtained in the ordered inhomogeneous random graph model.

\subsection{Hill and Pickands estimators}

In this section, we prove an approximation of the degree-based Hill and Pickands estimators by their weight-based counterparts. The Hill estimator based on the upper $k(n)$ degrees is given by
\begin{align}
\label{eq:hilld}
\hat{\gamma}^\text{Hill}_D(n) = \frac{1}{k(n) - 1}\sum_{i = 1}^{k(n) - 1}\log\left(D_{(i)}(n)/D_{(k(n))}(n) \right).
\end{align}
The weight-based Hill estimator, $\hat{\gamma}^\text{Hill}_W(n)$, is defined by replacing $D_{(i)}(n)$ in \eqref{eq:hilld} with $W_{(i)}(n)$ for $i \in [k(n)]$. The Pickands estimator based on the upper $4k(n)$ degrees is given by
\begin{align*}
\hat{\gamma}^\text{Pick}_D(n) = \frac{1}{\log 2} \log\left( \frac{D_{(k(n))}(n) - D_{(2k(n))}(n)}{D_{(2k(n))}(n) - D_{(4k(n))}(n)}  \right),
\end{align*}
with its weight-based counterpart defined similarly. In either case, the goal is to prove a statement of the form 
\begin{align}
\label{eq:hill_approx}
\sqrt{k(n) }\left|\hat{\gamma}_D(n) - \hat{\gamma}_W(n) \right| \xrightarrow{p} 0, \qquad \text{as } n \rightarrow \infty.
\end{align} 
For the Hill and Pickands estimators the approximation is handled similarly since both involve a difference of logarithms. A formal statement of \eqref{eq:hill_approx} is given in Lemma \ref{lem:hill_approx}. Asymptotic normality of the centered and scaled tail index estimators is again ensured through the second order condition on $U$. That is, by assuming \eqref{eq:2rv} and that $k(n)A(n/k_n) \rightarrow 0$, we have that for $\ell \in \{\text{Hill}, \text{Pick} \}$
\begin{align}
\label{eq:asympnormHP}
\sqrt{k(n)}\left( \hat{\gamma}^\ell_W(n) - \gamma \right) \Rightarrow N\left(0, \tau^{\ell} \right)
\end{align}
as $n \rightarrow \infty$ where $\tau^{\text{Hill}} = \gamma^2$ and 
\begin{align*}
\tau^{\text{Pick}} = \frac{\gamma^2(2^{2\gamma + 1} + 1)}{4(\log 2)^2(2^\gamma - 1)^2}. 
\end{align*}
See Theorems 3.2.5 and 3.3.5 of \cite{haan2006extreme} for more details. Since $2(2^\gamma - 1)^2< 2^{2\gamma + 1} + 1$ for $\gamma > 0$, it is easily seen that $\tau^{\text{Pick}} > \gamma^2/2(\log 2)^2 > \tau^{\text{Hill}}$. The statements \eqref{eq:hill_approx} and \eqref{eq:asympnormHP}, in conjuction with Slutsky's theorem, give asymptotic normality of the tail index estimators based on the degrees. Asymptotic normality of the Hill and Pickands estimators for Norros-Reittu and Chung-Lu random graph models is presented in Theorem \ref{thm:hill}. Theorem \ref{thm:hill} is proved in Section \ref{sec:tail_approx}. 

\begin{theorem}
\label{thm:hill}
Assume that $U$ satisfies \eqref{eq:rv} and \eqref{eq:Ulower}.
Suppose that the network is generated under either a Norros-Reittu model with $\alpha > 0$ or Chung-Lu model with $\alpha > 2$. Suppose the second order condition \eqref{eq:2rv} holds with  $k(n) \rightarrow \infty$, $k(n) \log^{\alpha}(n) /n^{\frac{1}{4\alpha + 1}} \rightarrow 0$ and $k(n)A(n/k(n)) \rightarrow 0$ as $n \rightarrow \infty$. Then for $\ell \in \{\textup{Hill}, \textup{Pick} \}$
\begin{align}
\label{eq:asymp_normD}
\sqrt{k(n)}\left( \hat{\gamma}^\ell_D(n) - \gamma \right) \Rightarrow N\left(0, \tau^{\ell} \right),  \qquad \text{as } n \rightarrow \infty.
\end{align}
\end{theorem}

The key finding in Theorem \ref{thm:hill} is that despite the degree dependence present in IRGs, the tail index estimators are asymptotically normal with the same asymptotic variance as in the iid case. Though, as we discuss in the conclusion, we do not expect the same result to hold for other scale-free random graph models. 

\subsection{Probability weighted moment estimator}

The degree-based PWM estimator based on the upper $k(n)$ degrees is given by
\begin{align*}
\hat{\gamma}^\text{PWM}_D(n) = \frac{\hat{I}^{(1)}_{D}(n) - 4\hat{I}^{(2)}_{D}(n)}{\hat{I}^{(1)}_{D}(n) - 2\hat{I}^{(2)}_{D}(n)},
\end{align*}
where for $q = 1, 2$ 
\begin{align}
\label{eq:pwm}
\hat{I}^{(q)}_{D}(n) = \frac{1}{k(n) - 1} \sum_{i = 1}^{k(n) - 1}\left(\frac{i}{k(n) - 1}\right)^{q-1}\left(D_{(i)}(n) - D_{(k(n))}(n) \right).
\end{align}
As before, the weight-based version is defined by replacing the degree order statistics by the weight order statistics. Typically, asymptotic normality of the PWM estimator is achieved by showing that the probability weighted moments in \eqref{eq:pwm} are jointly asymptotically normal and then applying the delta method \cite{haan2006extreme}. Hence, we assume that the weighted-based probability weighted moments, $(\hat{I}^{(1)}_{W}(n), \hat{I}^{(2)}_{W}(n))$, are asymptotically normal and then approximate the degree-based moments by their weighted counterparts. In particular, we prove in Lemma \ref{lem:pwm_approx} that for $q \in \{1, 2\}$
\begin{align}
\label{eq:pwm_approx}
\begin{split}
\sqrt{k(n)}\left|\hat{I}^{(q)}_{D}(n) - \hat{I}^{(q)}_{W}(n) \right| \xrightarrow{p}& 0
\end{split}
\end{align}
as $n \rightarrow \infty$. Asymptotic normality of the centered and scaled  $(\hat{I}^{(1)}_{W}(n), \hat{I}^{(2)}_{W}(n))$ is again ensured through the second order condition on $U$. That is, by assuming \eqref{eq:2rv} and that $k(n)A(n/k_n) \rightarrow 0$, one may achieve that
\begin{align}
\label{eq:pwms:norm}
\sqrt{k(n)}\left( 
\begin{pmatrix} 
\hat{I}^{(1)}_{W}(n)/a(n/k(n)) \\ 
\hat{I}^{(2)}_{W}(n)/a(n/k(n)) 
\end{pmatrix}
- 
\begin{pmatrix} 
\frac{1}{1 -\gamma} \\ 
\frac{1}{2(2 -\gamma)} 
\end{pmatrix}
\right) 
\Rightarrow 
N\left( 
\begin{pmatrix} 
0 \\ 
0 
\end{pmatrix}
,
\begin{pmatrix} 
\sigma_{11} & \sigma_{12} \\
\sigma_{21} & \sigma_{22}
\end{pmatrix}
\right),
\end{align}
with $a(t) \equiv \gamma U(t)$ for $t > 1$ and for $q, r \in \{1, 2\}$
\begin{align*}
\sigma_{qr} = \frac{1}{q(q - \gamma)r(r-\gamma)}\left( \frac{qr}{q + r - 1 - 2\gamma} + \gamma^2 \right). 
\end{align*}
For more details see either Proposition 4.1 of \cite{de2024bootstrapping} or Theorem 3.6.1 of \cite{haan2006extreme}, both of which rely on Theorem 2.4.8 of \cite{haan2006extreme} in our setting. We now present the asymptotic normality of the degree-based PWM estimator.
\begin{theorem}
\label{thm:pwm_norm}
Assume that $U$ satisfies \eqref{eq:rv} and \eqref{eq:Ulower}.
Let $\alpha > 2$ and suppose the network is generated under either the Norros-Reittu or Chung-Lu model. Further suppose that the second order condition \eqref{eq:2rv} holds with $k(n) \rightarrow \infty$, $k(n)\log^\alpha (n)/n^{\frac{1}{4\alpha + 1}} \rightarrow 0$ and $k(n)A(n/k(n))\rightarrow 0$ as $n \rightarrow \infty$. Then
\begin{align*}
\sqrt{k(n)} \left(\hat{\gamma}^\textup{PWM}_D(n) - \gamma\right) \Rightarrow N(0, \tau^\textup{PWM}),
\end{align*}
as $n \rightarrow \infty$ where 
\begin{align*}
\tau^\textup{PWM} = \frac{(1 - \gamma)(2 - \gamma)^2(1 - \gamma + 2\gamma^2)}{(1 - 2\gamma)(3 - 2\gamma)}.
\end{align*}
\end{theorem}
Theorem \ref{thm:hill} is proved in Section \ref{sec:tail_approx}. Using the fact that $2(2 - \gamma) > 3 - 2\gamma$ and $1 - \gamma > 1 - 2\gamma$, it is trivially seen that $\tau^\textup{PWM} > \tau^\textup{Hill}$ for $\gamma < 1/2$. The asymptotic variances $\tau^\textup{PWM}$ and $\tau^\textup{Pick}$ are more difficult to compare, though plots show that $\tau^\textup{PWM} < \tau^\textup{Pick}$ for approximately $\gamma < 0.4142$ while  $\tau^\textup{PWM} \geq \tau^\textup{Pick}$ otherwise. Like the Hill and Pickands estimators, the degree-based PWM estimator adopts the asymptotic variance from the iid setting.

\section{Simulations}\label{sec:sim}

We now present simulation studies verifying the findings in Theorems \ref{thm:order}, \ref{thm:hill} and \ref{thm:pwm_norm}. We first numerically examine whether $k(n)$ can grow any faster than the rate of $n^{\frac{1}{4\alpha + 1}}/\log^\alpha(n)$ and still achieve alignment of order statistics in \eqref{eq:ord_aling}. Secondly, we evaluate how well the asymptotic normality of the tail index estimators holds for different choices of $k(n)$ and $\alpha$ across the two random graph models. The latter simulations shed light on issues with tail index estimation for network data. 

\subsection{Empirical alignment of upper order statistics}

Define the first index for which the alignment of upper order statistics does not hold for a given sample as
\begin{align*}
K(n) := \inf\left\lbrace m \in [n]: D_{(m)}(n) \neq D_{m}(n) \right\rbrace, 
\end{align*}
and clearly, $K(n)$ is supported on $[n]$. In order for the alignment of order statistics to hold with high probability for a given $k(n)$ sequence, we would expect $K(n)$ to grow slightly faster than $k(n)$. If the rate for $k(n)$ derived in Theorem \ref{thm:order} is tight, then we may also expect $K(n)$ and $k(n)$ to be approximately of the same order. 

In order to evaluate how the empirical distribution of $K(n)$ scales with $n$, we consider the following simulation set up. Suppose $W_{[n]}(n)$ are the order statistics of $n$ independent samples from either a (a) Pareto distribution supported on $[2, \infty)$ with $\alpha = 1$, or (b) the distribution function $1 - F(x) = (2+x)/x^2$, $x > 2$. For each sample size $n \in \{2^{10}, 2^{11}, \dots, 2^{20} \}$, we generate $1{,}000$ Norros-Reittu random graphs with weight distributions (a) and (b) and compute $K(n)$ for each realization. Here, (a) and (b) are chosen so that the only obvious difference between the two distributions is the presence of the slowly-varying part of $1 - F(x)$ in (b). 

\begin{figure}[h]
\centering
\includegraphics[width=\textwidth]{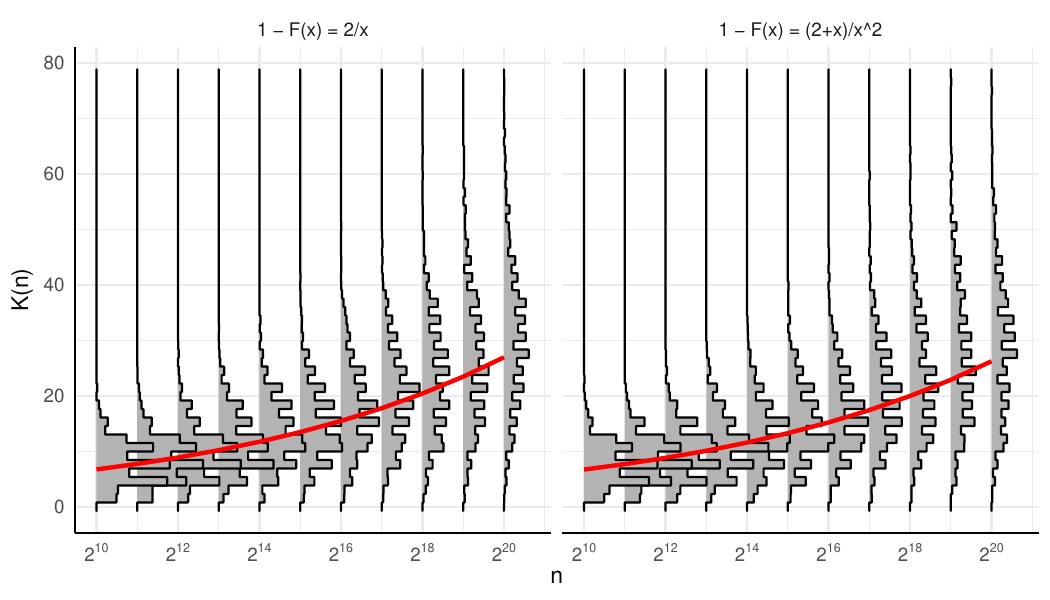}
\caption{Histograms of 1000 realizations of $K(n)$ for across sample sizes $n \in \{2^{10}, 2^{11}, \dots, 2^{20}\}$ with weight distributions (a) $1 - F(x) = 2/x$ and (b) $1 - F(x) = (2 + x)/x^2$, $x > 2$. The red line corresponds to the curves $y_a = 0.7430 n^{0.2004}$ and $y_b = 0.7717n^{0.1970}$ where the parameters are estimated by regressing the $\log_2(K(n))$ realizations against $\log_2(n)$ for scenarios (a) and (b), respectively.}\label{fig:k}
\end{figure}

Histograms of $K(n)$ for each sample size are presented in Figure \ref{fig:k}. In red we plot the curves $y_a = 0.7430 n^{0.2004}$ and $y_b = 0.7717n^{0.1970}$ where the coefficients and powers are obtained by regressing the $\log_2(K(n))$ realizations against $\log_2(n)$ for the scenarios (a) and (b), respectively. In either case, $K(n)$ grows approximately proportional to $n^{0.2}$, the theoretical upper bound obtained in Theorem \ref{thm:order} for $k(n)$, ignoring logarithmic terms. This indicates that our theory reveals a reasonably tight upper bound on the growth rate of $k(n)$.

\subsection{Sensitivity to the choice of $k(n)$}

We now evaluate how sensitive the normal approximations to the distributions of the tail index estimators are to the choice on $k(n)$ across a range of $\alpha$ values. In each simulation setting,  we consider networks of $n = 2{,}000{,}000$ nodes with $W_{[n]}(n)$ being the order statistics of $n$ independent samples from a Burr distribution with $1 - F(x) = (1 + x^{\alpha})^{-1}$. This distribution satisfies the second order condition with $\rho = -1$. We generate $1{,}000$ replicates of Norros-Reittu and Chung-Lu random graphs across a range of $\alpha$ values. For the Norros-Reittu model we consider $\alpha \in \{1.25, 1.5, 1.75 \}$ while for the Chung-Lu model we let $\alpha \in \{2.25, 2.5, 2.75\}$, the latter choice being consistent with the requirement that $\alpha > 2$ in our theory. For each network replicate, we compute the Hill estimator for $k(n) \in \{32, 64, 128, 256\}$ and the Pickands estimator for $k(n) \in \{8, 16, 32, 64\}$. This choice is made so that the last order statistic used by the Pickands estimator, the $4k(n)$-th largest, is the same as the last order statistic used by the Hill estimator. We only compute the PWM estimator on the Chung-Lu random graphs since $\alpha > 2$ is required to achieve asymptotic normality, but evaluate the estimator over a wider range of $k(n) \in \{32, 64, \dots, 2048\}$. Figures \ref{fig:hill}, \ref{fig:pick} and \ref{fig:pwm}, display the empirical distributions of the centered and scaled tail index estimators, along with their expected asymptotic distributions.    

\begin{figure}[h]
\centering
\includegraphics[width=\textwidth]{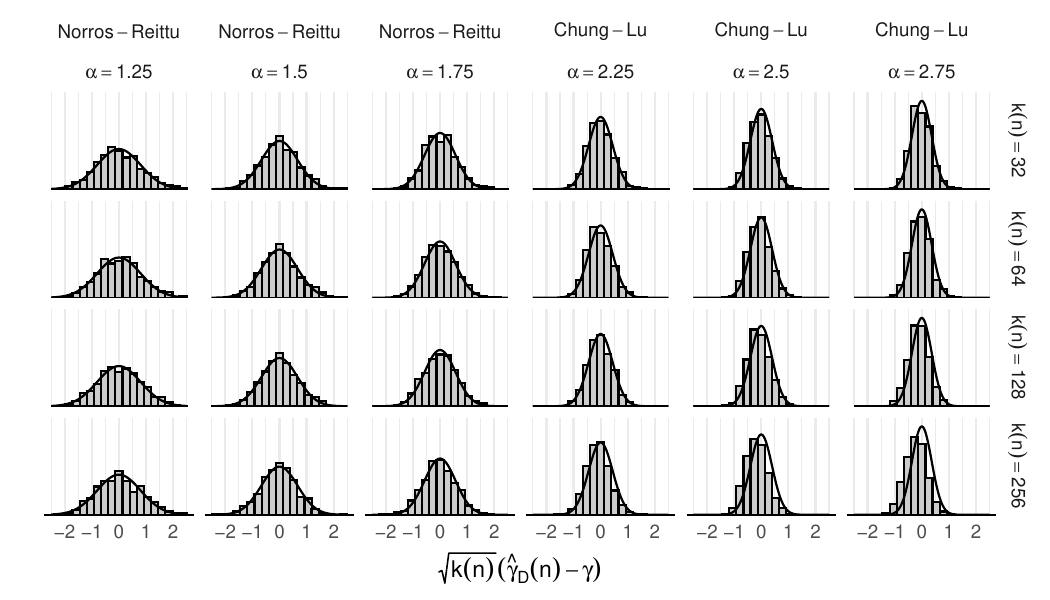}
\caption{Histograms of centered and scaled Hill estimators. For each $\alpha \in \{1.25, 1.5, 1.75 \}$ we simulate $1{,}000$ Norros-Reittu graphs and for each $\alpha \in \{2.25, 2.5, 2.75 \}$ we simulate $1{,}000$ Chung-Lu graphs. For each multigraph we compute the Hill estimator for $k(n) \in \{32, 64, 128, 256\}$. }\label{fig:hill}
\end{figure}

For the Hill estimator, histograms indicate that in order to retain asymptotic normality of the Hill estimator as $\alpha$ increases, it is required that $k(n)$ should decrease. This is consistent with our theory since we require $k(n) = o(n^{\frac{1}{4\alpha + 1}}/\log^\alpha(n))$ to achieve asymptotic normality of the Hill estimator in Theorem \ref{thm:hill}. We also note that for larger values of $\alpha$, ties among the degrees form in finite-sample situations. It is well-known that the Hill estimator behaves erratically in the presence of ties which, along with our theory, partially explains the non-normal behavior for large values of $\alpha$ and $k(n)$ \cite{matsui2013estimation, resnick2007heavy}. 

\begin{figure}[h]
\centering
\includegraphics[width=\textwidth]{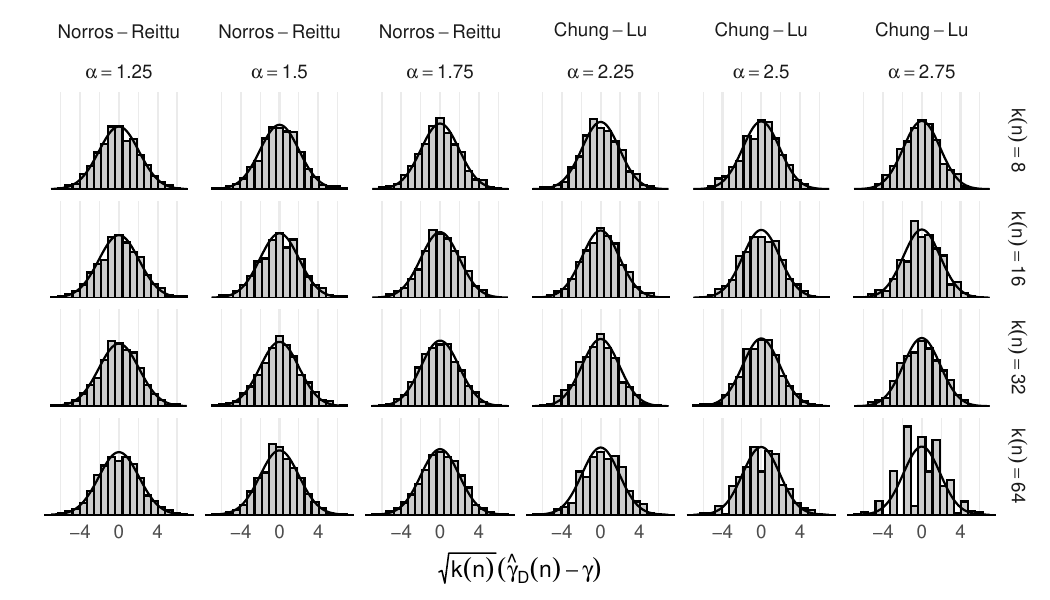}
\caption{Histograms of centered and scaled Pickands estimators. For each $\alpha \in \{1.25, 1.5, 1.75 \}$ we simulate $1{,}000$ Norros-Reittu graphs and for each $\alpha \in \{2.25, 2.5, 2.75 \}$ we simulate $1{,}000$ Chung-Lu graphs. For each multigraph we compute the Pickands estimator for $k(n) \in \{8, 16, 32, 64\}$. }\label{fig:pick}
\end{figure}

From Figure~\ref{fig:pick}, we see that the Pickands estimator behaves less sensitively to the choice of $k(n)$, at least among the smaller values of $k(n)$. Once $k(n) = 64$ and ties are more present, however, the erratic behavior of the Pickands estimator becomes more apparent.

\begin{figure}[h]
\centering
\includegraphics[width=0.6\textwidth]{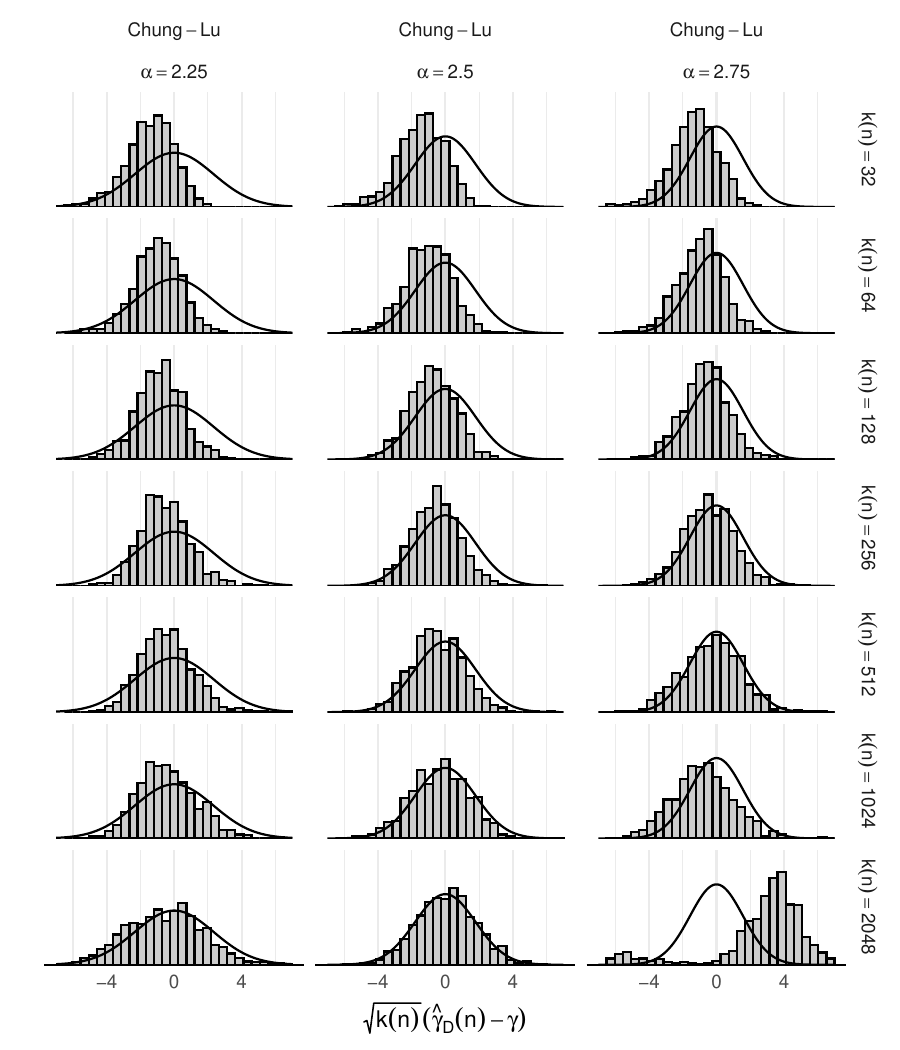}
\caption{Histograms of centered and scaled PWM estimators. For each $\alpha \in \{2.25, 2.5, 2.75 \}$ we simulate $1{,}000$ Chung-Lu graphs. For each multigraph we compute the PWM estimator for $k(n) \in \{32, 64, 128, 256, 512, 1024, 2048\}$. }\label{fig:pwm}
\end{figure}

The histograms in Figure~\ref{fig:pwm} indicate that the distributional approximation for the PWM estimator is more sensitive to the choice of $k(n)$. In fact, asymptotic normality is more reliably achieved for larger values of $k(n)$, at least until ties among the upper degrees become more frequent. Larger choices of $k(n)$ providing more accurate normal approximations for the PWM estimator is consistent with other simulation studies in the literature \cite{de2024bootstrapping}. 

In applied settings, it is common to compare tail index estimates across a variety of methods. Our simulations indicate that a choosing a common $k(n)$ across each method is not appropriate. Instead, as done in \cite{voitalov2019scale}, each tail index estimate should be computed with a personalized threshold that is estimated through a data-driven procedure.

\section{Conclusion}\label{sec:conclusion}

In this paper, we have taken a vital first step towards rigorously justifying statistical procedures for inferring the tail behavior of degree distributions from scale-free random networks. This is achieved by leveraging and further probing the relationship between the large weights and large degrees in inhomogeneous random graphs. By approximating the upper degree order statistics by their corresponding weight order statistics, we prove the asymptotic normality of three tail index estimators, each of which exhibited behavior analogous to the iid setting. We comment that this behavior likely only holds for the class of IRGs. Simulations not presented here indicate that tail index estimators computed on network models with temporal dependence, such as the preferential attachment network, do not reflect the same asymptotic behavior. 

The alignment of upper order statistics provides a general technique for inspecting the performance of extreme-value methods when applied to the degrees from scale-free IRGs. As mentioned earlier, a matter that hinders the application of tail index estimators to real-world networks is choosing $k(n)$. In future work, we aim to prove consistency of a double bootstrap procedure used to estimate the optimal choice of $k(n)$ \cite{danielsson2001using}. Simulation results indicate that this procedure performs well for scale-free networks, especially in comparison to other popular methods \cite{voitalov2019scale}. Moreover, we may finally provide a goodness-of-fit test for the scale-free model, the absence of which has led to much controversy in the network science community \cite{broido2019scale, holme2019rare, voitalov2019scale}. In particular, we envision extending the tests of \cite{dietrich2002testing} and \cite{husler2006testing} to scale-free networks.

\section{Proofs}\label{sec:proof}

\subsection{Proof of Lemma \ref{lem:perm}}\label{sec:perm}

\begin{proof}[Proof of Lemma \ref{lem:perm}]
Fix constant weights $\{w_i\}_{i = 1}^n$ with $l(n) = \sum_{i = 1}^n w_i$. Let $p$ be the conditional the mass function governing either \eqref{eq:NR} and \eqref{eq:CL}. That is, under the Norros-Reittu model
\begin{align*}
p(\{ a_{ij}\}_{1 \leq i \leq j \leq n } \mid w_1, \dots, w_n) =& \prod_{1 \leq i \leq j \leq n} e^{-w_iw_j/l(n)}(w_iw_j/l(n))^{a_{ij}}/a_{ij}!,
\end{align*}
where the $a_{ij} \in \mathbb{Z}_+$ for $1 \leq i \leq j \leq n$. On the other hand, under the Chung-Lu model
\begin{align*}
p(\{ a_{ij}\}_{1 \leq i \leq j \leq n } \mid w_1, \dots, w_n) =& \prod_{1 \leq i \leq j \leq n} (w_iw_j/l(n) \wedge 1)^{a_{ij}}(1 - w_iw_j/l(n) \wedge 1)^{1 - a_{ij}},
\end{align*}
where the $a_{ij} \in \{0, 1\}$ for $1 \leq i \leq j \leq n$. Then we have that
\begin{align*}
\mathbb{P}\big( \{&A_{\pi(i)\pi(j)}(n) \}_{1 \leq i \leq j \leq n } = \{ a_{ij}\}_{1 \leq i \leq j \leq n } \big) \\
=& \frac{1}{n!}\sum_{\sigma \in S_n} \mathbb{P}\left( \{A_{\sigma(i)\sigma(j)}(n) \}_{1 \leq i \leq j \leq n } = \{ a_{ij}\}_{1 \leq i \leq j \leq n } \right) \\
=& \frac{1}{n!}\sum_{\sigma \in S_n} \mathbb{E}\left[ \mathbb{P}\left( \{A_{\sigma(i)\sigma(j)}(n) \}_{1 \leq i \leq j \leq n } = \{ a_{ij}\}_{1 \leq i \leq j \leq n } \mid \{W_{(i)}(n)\}_{i \in [n]} \right) \right] \\
=& \frac{1}{n!}\sum_{\sigma \in S_n} \mathbb{E}\left[ p\left(\{ a_{ij}\}_{1 \leq i \leq j \leq n } \mid W_{(\sigma(1))}(n), \dots, W_{(\sigma(n))}(n) \right) \right] \\
=& \mathbb{E}\left[ \frac{1}{n!}\sum_{\sigma \in S_n} p\left(\{ a_{ij}\}_{1 \leq i \leq j \leq n } \mid W_{(\sigma(1))}(n), \dots, W_{(\sigma(n))}(n) \right) \right], \\
\intertext{and using that fact that the conditional distribution of $(W_1, \dots, W_n)$ given $(W_{(1)}(n), \dots, W_{(n)}(n))$ is uniform over all permutations of the vector $(W_{(1)}(n), \dots, W_{(n)}(n))$, we have that }
=& \mathbb{E}\left[ \mathbb{E}\left[ p\left(\{ a_{ij}\}_{1 \leq i \leq j \leq n } \mid W_1, \dots, W_n \right)  \mid W_{(1)}(n), \dots, W_{(n)}(n) \right] \right] \\
=& \mathbb{E}\left[  p\left(\{ a_{ij}\}_{1 \leq i \leq j \leq n } \mid W_1, \dots, W_n \right) \right].
\end{align*}
The last quantity is the mass function associated with an inhomogeneous random graph and hence the proof is complete.
\end{proof}

\subsection{Order statistic bounds}

In this section we develop some basic upper and lower bounds for the order statistics $W_{(i)}(n)$, $i \in [n]$. Concentration inequalities are first developed for exponential order statistics, and then translated to upper and lower bounds for $W_{(i)}(n)$, $i \in [n]$ using the distributional representation in \eqref{eq:orddist}. We first recall R\'enyi's representation for exponential order statistics \cite{renyi1953theory}. 

\begin{theorem}
\label{thm:renyi}
Suppose $Y_1, Y_2, \dots, Y_n$ are independent unit rate exponential random variables. Then
\begin{align*}
\left( Y_{(1)}(n), Y_{(2)}(n), \dots, Y_{(n)}(n) \right) \overset{d}{=} \left( \sum_{j = 1}^n \frac{E_j}{j}, \sum_{j = 2}^n \frac{E_j}{j}, \dots, \frac{E_n}{n} \right).
\end{align*}
for independent unit rate exponential random variables $E_1, E_2, \dots, E_n$.
\end{theorem}

For the R\'enyi representation components, we have the following concentration inequality. Note that since $i \sum_{j = i}^n \frac{1}{j^2} \leq \frac{\pi^2}{6}$, a sufficient condition for $\epsilon > 2i \sum_{j = i}^n \frac{1}{j^2}$ for all $i \in [n]$ is that $\epsilon > \frac{\pi^2}{3}$. 
\begin{lemma}
\label{lem:conc_exp}
Let $E_1, \dots, E_n$ be iid unit rate exponential random variables. Then for any $i \in [n]$
\begin{align*}
\mathbb{P}\left( \left| \sum_{j = i}^n \frac{(E_j - 1)}{j} \right| > \epsilon \right) \leq 
\begin{cases}
2\exp\left\lbrace - \frac{\epsilon^2}{8 \sum_{j = i}^n \frac{1}{j^2}}\right\rbrace, \qquad &\text{if } \epsilon \leq 2i \sum_{j = i}^n \frac{1}{j^2}, \\
2\exp\left\lbrace - \frac{i}{4} \epsilon \right\rbrace, \qquad &\text{if } \epsilon > 2i \sum_{j = i}^n \frac{1}{j^2}.
\end{cases}
\end{align*}
\end{lemma}

\begin{proof}
We first focus on bounding the moment generating function of a centered exponential random variable. Towards this end, suppose $X$ is an exponential random variable with unit rate. It is well-known that for any $m \in \mathbb{N}$
\begin{align*}
\mathbb{E}\left[\left(X - 1 \right)^m\right] = m! \sum_{j = 0}^m \frac{(-1)^j}{j!}.
\end{align*}
Hence $\left|\mathbb{E}\left[(X - 1)^m\right] \right| \leq m!$. Let $t \in (0, 1/2]$. Then
\begin{align*}
\mathbb{E}\left[e^{t(X - 1)} \right] = 1 + \sum_{m = 2}^\infty \frac{\mathbb{E}\left[(X - 1)^m \right] t^m}{m!} \leq  1 + \sum_{m = 2}^\infty t^m =& \frac{1}{1 - t} - t \\
=& 1 + \frac{t^2}{1 - t} \\
\leq& 1 + 2t^2 \\
\leq& e^{2t^2}.
\end{align*}
With the moment generating function bound in hand, we now turn to proving the concentration inequality. Let $s \in (0, i/2]$ and let $\epsilon > 0$. Apply Chernoff's inequality and independence of $E_j$, $j \in [n]$, to obtain that
\begin{align*}
\mathbb{P}\left( \sum_{j = i}^n \frac{\left(E_j - 1 \right)}{j} > \epsilon \right) \leq e^{-s\epsilon}\prod_{j = i}^n \mathbb{E}\left[e^{s\frac{(E_j - 1)}{j}} \right] \leq \exp\left\lbrace -s\epsilon + 2 s^2 \sum_{j = i}^n \frac{1}{j^2} \right\rbrace.
\end{align*}
With the constraint $s \leq i/2$, the above bound is minimized at
\begin{align*}
s^\star = \frac{\epsilon}{4 \sum_{j = i}^n \frac{1}{j^2}} \wedge \frac{i}{2}. 
\end{align*}
Hence 
\begin{align*}
\exp\left\lbrace -s^\star\epsilon + 2 (s^\star)^2 \sum_{j = i}^n \frac{1}{j^2}\right\rbrace \leq \exp\left\lbrace -s^\star\epsilon + \frac{s^\star\epsilon}{2} \right\rbrace =& \exp\left\lbrace -\frac{s^\star\epsilon}{2} \right\rbrace.
\end{align*}
Therefore
\begin{align*}
\mathbb{P}\left( \sum_{j = i}^n \frac{(E_j - 1)}{j}  > \epsilon \right) \leq 
\begin{cases}
\exp\left\lbrace - \frac{\epsilon^2}{8 \sum_{j = i}^n \frac{1}{j^2}}\right\rbrace, \qquad &\text{if } \epsilon \leq 2i \sum_{j = i}^n \frac{1}{j^2}, \\
\exp\left\lbrace - \frac{i}{4} \epsilon \right\rbrace, \qquad &\text{if } \epsilon > 2i \sum_{j = i}^n \frac{1}{j^2}.
\end{cases}
\end{align*}
The bound for $\mathbb{P}\left(  \sum_{j = i}^n \frac{(E_j - 1)}{j}  < -\epsilon \right)$ is achieved similarly and thus the proof is omitted.
\end{proof}

Using Lemma \ref{lem:conc_exp}, we may develop high probability upper and lower bounds for $W_{(i)}(n)$, $i \in [n]$. Such bounds will help simplify proofs quantifying the gaps between consecutive order statistics.

\begin{lemma}
\label{lem:bounds}
Let $i \in [n]$ and fix $\epsilon > \frac{2i}{\alpha} \sum_{j = i}^n \frac{1}{j^2}$. Then,
\begin{align}
\label{eq:upper}
&\mathbb{P}\left( W_{(i)}(n) >  U\left(e^{\sum_{j = i}^n \frac{1}{j} + \epsilon}\right)  \right) \leq 2\exp\left\lbrace - \frac{i}{4} \epsilon \right\rbrace, \\
\label{eq:lower}
&\mathbb{P}\left( W_{(i)}(n) <  U\left(e^{\sum_{j = i}^n \frac{1}{j} - \epsilon}\right)  \right) \leq 2\exp\left\lbrace - \frac{i}{4} \epsilon \right\rbrace.
\end{align}
\end{lemma}

\begin{proof}
We only prove \eqref{eq:upper} since \eqref{eq:lower} is proved in a similar manner. Recall the representation in \eqref{eq:orddist} and employ montonicity of $U$ to achieve that
\begin{align*}
\mathbb{P}\left( U\left(e^{\sum_{j = 1}^n \frac{E_j}{j}}\right) >  U\left(e^{\sum_{j = i}^n \frac{1}{j} + \epsilon}\right) \right) =& \mathbb{P}\left( e^{\sum_{j = 1}^n \frac{E_j}{j}} >  e^{\sum_{j = i}^n \frac{1}{j} + \epsilon} \right)\\
=& \mathbb{P}\left( e^{\sum_{j = i}^n \frac{(E_j - 1)}{j}} > e^{\epsilon} \right) \\
=& \mathbb{P}\left( \sum_{j = i}^n \frac{(E_j - 1)}{j} > \epsilon \right) \\
\leq&  2\exp\left\lbrace - \frac{i}{4} \epsilon \right\rbrace,
\end{align*}
where we have appealed to Lemma \ref{lem:conc_exp}, given that $\epsilon > \frac{2i}{\alpha} \sum_{j = i}^n \frac{1}{j^2}$. 
\end{proof}
We now state a simple result that will prove useful in simplifying later concentration inequalities involving the degrees $D_i(n)$, $i \in [k(n)]$. 
\begin{lemma}
\label{lem:log}
Assume that $U$ satisfies \eqref{eq:rv} and \eqref{eq:Ulower}. Fix $L > 0$. Suppose $k(n) \log^\alpha(n)/n \rightarrow 0$ as $n \rightarrow \infty$. Then as $n \rightarrow \infty$
\begin{align*}
\mathbb{P}\left( W_{(k(n))}(n) > L \log(n) \right) \rightarrow 1.
\end{align*}
\end{lemma}

\begin{proof}
From Theorem 4.2 of \cite{resnick2007heavy}, we have that $W_{(k(n))}(n)/U(n/k(n)) \xrightarrow{p} 1$ and thus for any $\delta  \in (0, 1)$
\begin{align*}
\mathbb{P}\left(W_{(k(n))}(n) > (1 - \delta) U(n/k(n)) \right) \rightarrow 1, \qquad \text{as }n \rightarrow \infty.
\end{align*}
Hence, it suffices that $U(n/k(n))/\log(n) \rightarrow \infty$, which is clear since from \eqref{eq:Ulower} we have that $U(n/k(n)) \geq (1 - \eta)\left(\frac{n}{t_0 k(n)}\right)^{1/\alpha}U(t_0)$ for $n$ sufficiently large. 
\end{proof}

\subsection{Order statistic spacings}\label{sec:space_proof}

In order to ensure that the alignment of upper order statistics is achieved, we require that the degrees are in general close to their weights while consecutive weights are sufficiently far apart. This section deals with the latter requirement. That is, we establish an upper bound on the number of upper order statistics such that consecutive weights repel at an appropriate rate.

\begin{lemma}
\label{lem:w_repel}
Assume that $U$ satisfies \eqref{eq:rv} and \eqref{eq:Ulower}. Suppose $k(n)$ is such that $k(n) \rightarrow \infty$ and $k(n) \log^{\alpha}(n) /n^{\frac{1}{4\alpha + 1}} \rightarrow 0$ as $n \rightarrow \infty$. Then
\begin{align*}
\mathbb{P}&\left( \forall i \in [k(n)],  W_{(i)}(n) - W_{(i + 1)}(n) >  2\sqrt{5\log (n) W_{(i)}(n)}  \right) \rightarrow 1.
\end{align*} 
\end{lemma}

\begin{proof}
Define the event
\begin{align*}
T(n) = \left\lbrace e^{\sum_{j = k(n) + 1}^n\frac{E_j}{j}} \geq t_0 \right\rbrace. 
\end{align*}
Clearly, $P(T(n)) \rightarrow 1$ as $n \rightarrow \infty$. Thus it suffices to prove that as $n \rightarrow \infty$
\begin{align}
\label{eq:spacinggoal}
\mathbb{P}\left( \exists i \in [k(n)] : U\left(e^{\sum_{j = i}^n\frac{E_j}{j}} \right) -  U\left(e^{\sum_{j = i + 1}^n\frac{E_j}{j}} \right) < 2\sqrt{5\log (n) U\left(e^{\sum_{j = i}^n\frac{E_j}{j}} \right)}, T(n) \right) \rightarrow 0.
\end{align}
Note that when $t \geq t_0$, we may employ \eqref{eq:Ulower} to state that for $x \geq 1$
\begin{align*}
U(tx) - U(t) =& U(tx) - x^{-\gamma}U(tx) + x^{-\gamma}U(tx) - U(t) \\
=& U(tx)\left( \left(1 - x^{-\gamma}\right) +  \frac{U(t)}{U(tx)}\left(\frac{\ell(tx)}{\ell(t)} - 1 \right)\right) \\
\geq& U(tx)\left( \left(1 - x^{-\gamma}\right) -\eta \frac{U(t)}{U(tx)}\left(1 - x^{-\gamma} \right)\right) \\
\geq&  (1 - \eta)U(tx)\left(1 - x^{-\gamma}\right),
\end{align*}
where in the last line we have employed the fact that $U(t)/U(tx) \leq 1$. Hence, on the set $T(n)$
\begin{align*}
 U\left(e^{\sum_{j = i}^n\frac{E_j}{j}} \right) -  U\left(e^{\sum_{j = i + 1}^n\frac{E_j}{j}} \right) \geq& (1 - \eta)U\left(e^{\sum_{j = i}^n\frac{E_j}{j}} \right)\left(1 - e^{-\frac{E_i}{\alpha i}} \right).
\end{align*}
Thus
\begin{align*}
\sum_{i = 1}^{k(n)}&\mathbb{P}\left(  U\left(e^{\sum_{j = i}^n\frac{E_j}{j}} \right) -  U\left(e^{\sum_{j = i + 1}^n\frac{E_j}{j}} \right) < 2\sqrt{5\log (n) U\left(e^{\sum_{j = i}^n\frac{E_j}{j}} \right)}, T(n)  \right) \\
\leq& \sum_{i = 1}^{k(n)}\mathbb{P}\left(  U\left(e^{\sum_{j = i}^n\frac{E_j}{j}} \right)\left(1 - e^{-\frac{E_i}{\alpha i}} \right) < \frac{2}{1 - \eta}\sqrt{5\log (n) U\left(e^{\sum_{j = i}^n\frac{E_j}{j}} \right)}  \right) \\
=& \sum_{i = 1}^{k(n)}\mathbb{P}\left(  1 - e^{-\frac{E_i}{\alpha i}} < \frac{2}{1 - \eta}\sqrt{5\log (n)/ U\left(e^{\sum_{j = i}^n\frac{E_j}{j}} \right)}  \right).
\end{align*}
Let $a(n)$ be a monotone increasing sequence tending towards infinity as $n \rightarrow \infty$ such that $a(n) =  o\left( n/\left(\log^\alpha(n) k^{4\alpha + 1}(n) \right) \right)$. We then have that for $i \in [k(n)]$
\begin{align}
\label{eq:andiv}
\frac{n}{a(n)i} \geq \frac{n}{a(n)k(n)} = \frac{n}{a(n)k^{4\alpha + 1}(n)\log^\alpha(n)} \log^\alpha(n) k^{4\alpha}(n) \rightarrow \infty.
\end{align}
We partition on the events
\begin{align*}
A_i(n) =& \left\lbrace U\left(e^{\sum_{j = i}^n\frac{E_j}{j}} \right) \geq U\left(\frac{1}{a(n)}\frac{n}{i} \right) \right\rbrace,
\end{align*}
for $i \in [k(n)]$. See that
\begin{align*}
\sum_{i = 1}^{k(n)}&\mathbb{P}\left( 1 - e^{-\frac{E_i}{\alpha i}}  < \frac{2}{1 - \eta}\sqrt{5 \log (n)/U\left(e^{\sum_{j = i}^n\frac{E_j}{j}} \right)}  \right) \\
\leq& \sum_{i = 1}^{k(n)}\mathbb{P}\left( 1 - e^{-\frac{E_i}{\alpha i}}  < \frac{2}{1 - \eta}\sqrt{5 \log (n)/U\left(e^{\sum_{j = i}^n\frac{E_j}{j}} \right)}, A_i(n) \right) + \sum_{i = 1}^{k(n)}\mathbb{P}(A^c_i(n)) \\
\leq& \sum_{i = 1}^{k(n)}\mathbb{P}\left( 1 - e^{-\frac{E_i}{\alpha i}}  < \frac{2}{1 - \eta}\sqrt{5 \log (n)/U\left(\frac{1}{a(n)} \frac{n}{i} \right)} \right) + \sum_{i = 1}^{k(n)}\mathbb{P}(A^c_i(n)) \\
=& \sum_{i = 1}^{k(n)}r_i(n) + \sum_{i = 1}^{k(n)}s_i(n). 
\end{align*}
Define $m_i(n) = \frac{2}{1 - \eta}\sqrt{5 \log (n)/U\left(\frac{1}{a(n)} \frac{n}{i} \right)}$. Note that for $i \in [k(n)]$, $m_i(n) \rightarrow 0$ since by  \eqref{eq:Ulower} and \eqref{eq:andiv}, for $n$ sufficiently large
\begin{align*}
U\left(\frac{1}{a(n)} \frac{n}{i} \right) \geq U\left(\frac{1}{a(n)} \frac{n}{k(n)} \right) \geq (1-\eta)U\left(\frac{n}{a(n)k^{4\alpha + 1}(n)\log^\alpha(n)}\right) \log(n) k^{4}(n).
\end{align*}
For $r_i(n)$, suppose that $n$ is large enough so that $m_{k(n)}(n) < 1$. Then
\begin{align*}
\sum_{i = 1}^{k(n)}r_i(n) =& \sum_{i = 1}^{k(n)}\mathbb{P}\left(e^{\frac{E_i}{\alpha i}}  < (1 - m_i(n))^{-1} \right) \\
=& \sum_{i = 1}^{k(n)}\left(1 - \left(1 - m_i(n) \right)^{\alpha i} \right), \\
\intertext{and by the mean value theorem}
\leq& \sum_{i = 1}^{\floor{1/\alpha}} \left(1 - \left( 1 - m_i(n) \right)^{\alpha i } \right) + \alpha \sum_{i = \floor{1/\alpha} + 1}^{k(n)} i m_i(n) \\
\leq& \floor{1/\alpha} \left(1 - \left( 1 - m_{k(n)}(n) \right) \right) + \alpha \sum_{i = \floor{1/\alpha} + 1}^{k(n)} i m_i(n) \\
=& \floor{1/\alpha} m_{k(n)}(n) + \alpha \sum_{i = \floor{1/\alpha} + 1}^{k(n)} i m_i(n). 
\end{align*}
By \eqref{eq:Ulower}, $U\left(\frac{1}{a(n)} \frac{n}{i} \right) = U\left(t_0 \frac{1}{a(n) t_0} \frac{n}{i} \right) \geq (1 - \eta)\left( \frac{1}{a(n) t_0} \frac{n}{i} \right)^{1/\alpha} U(t_0)$ for $n$ sufficiently large and hence
\begin{align*}
\sum_{i = 1}^{k(n)}r_i(n) \leq& \floor{1/\alpha} m_{k(n)}(n) + \alpha \sum_{i = \floor{1/\alpha} + 1}^{k(n)} i m_i(n) \\
=& \frac{2}{\sqrt{1 - \eta}}\floor{1/\alpha}\sqrt{5 t_0^{1/\alpha}\log (n)\left(a(n)\frac{k(n)}{n} \right)^{1/\alpha}/U(t_0)} \\
&+ \frac{2}{\sqrt{1 - \eta}}\alpha \sum_{i = \floor{1/\alpha} + 1}^{k(n)} i\sqrt{5 t_0^{1/\alpha}\log (n)\left(a(n)\frac{i}{n} \right)^{1/\alpha}/U(t_0)} \\
=&  O\left(\sqrt{\log (n)  \left(a(n) \frac{k(n)}{n} \right)^{1/\alpha}} \right) + O\left( \sqrt{a^{1/\alpha}(n)\log (n) \frac{k(n)^{\frac{4\alpha + 1}{\alpha}}}{n^{1/\alpha}}  }\right)
\end{align*}
Since $a(n) = o\left( n/\left(\log^{\alpha}(n) k^{4\alpha + 1}(n) \right) \right)$, $\sum_{i = 1}^{k(n)}r_i(n) \rightarrow 0$ as $n \rightarrow \infty$. We now turn to the terms $s_i(n)$. Define $\epsilon_i(n) = \sum_{j = i}^n \frac{1}{j} - \log (n/ia(n))$. Note that we may invoke Lemma \ref{lem:bounds} since $\epsilon_i(n) \geq \log(\frac{n+1}{n}a(n)) \rightarrow \infty$ and thus for $n$ sufficiently large
\begin{align*}
\sum_{i = 1}^{k(n)}s_i(n) =& \sum_{i = 1}^{k(n)}\mathbb{P}\left(U\left(e^{\sum_{j = i}^n\frac{E_j}{j}} \right) < U\left(\frac{1}{a(n)}\frac{n}{i} \right)\right) \\
=& \sum_{i = 1}^{k(n)}\mathbb{P}\left(U\left(e^{\sum_{j = i}^n\frac{E_j}{j}} \right) < U\left(e^{\sum_{j = i}^n\frac{1}{j} - \left(\sum_{j = i}^n\frac{1}{j} - \log(\frac{n}{a(n)i} \right)} \right) \right) \\
=&\sum_{i = 1}^{k(n)}2\exp\left\lbrace - \frac{i}{4} \log(a(n)) \right\rbrace, \\
\intertext{and for $n$ large enough so that $a(n) > 1$}
\leq& 2\left(-1 + \frac{1}{1 - a^{- \frac{1}{4}}(n)} \right) = 2\left(\frac{a^{- \frac{1}{4}}(n)}{1 - a^{- \frac{1}{4}}(n)} \right).
\end{align*}
Thus  $\sum_{i = 1}^{k(n)}s_i(n) \rightarrow 0$ and the proof is complete.
\end{proof}

\subsection{Degree concentration}\label{sec:conc_proof}

Recall that in order to achieve the alignment of the upper order statistics, we require that the degrees of nodes with large weights sufficiently concentrate around said weights. This section develops concentration inequalities for the degrees in the Norros-Reittu and Chung-Lu models. The first result concerns the rate at which $D_i(n)$ concentrates around $W_{(i)}(n)$ for $i \in [n]$ in the Norros-Reittu model. It is a straight-forward application of Poisson concentration results for large rates \citep[see][for example]{boucheron2013concentration, pollard2002user}
\begin{lemma}
\label{lem:Poi_conc}
For any $i \in [n]$, under the Norros-Reittu model
\begin{align}
\mathbb{P}\left( \left| D_i(n) - W_{(i)}(n) \right| \geq \sqrt{5\log (n) W_{(i)}(n)} \vee  5\log (n) \mid \{W_{(i)}(n)\}_{i \in [n]} \right) \leq n^{-5/4}.
\end{align}
\end{lemma}
\begin{proof}
See that for any $i \in [n]$
\begin{align*}
\mathbb{P}&\left( \left| D_i(n) - W_{(i)}(n) \right| \geq \sqrt{5\log (n) W_{(i)}(n)} \vee  5\log (n) \mid \{W_{(i)}(n)\}_{i \in [n]} \right) \\
\leq& \exp\left\lbrace - \frac{\left(  \sqrt{5\log (n) W_{(i)}(n)} \vee  5\log (n) \right)^2}{2\left(W_{(i)}(n) + \sqrt{5\log (n) W_{(i)}(n)} \vee  5\log (n) \right) } \right\rbrace.
\end{align*}
When $W_{(i)}(n) > 5\log (n)$, 
\begin{align*}
\frac{\left(  \sqrt{5\log (n) W_{(i)}(n)} \vee  5\log (n) \right)^2}{2\left(W_{(i)}(n) + \sqrt{5\log (n) W_{(i)}(n)} \vee  5\log (n) \right) } =& \frac{5\log (n) W_{(i)}(n)}{2\left(W_{(i)}(n) + \sqrt{5\log (n) W_{(i)}(n)}  \right) } \\
>&   \frac{5\log (n) W_{(i)}(n)}{4W_{(i)}(n) } \\
=& \frac{5}{4}\log(n).
\end{align*}
On the other hand, when  $W_{(i)}(n) \leq 5\log(n)$
\begin{align*}
\frac{\left(  \sqrt{5\log (n) W_{(i)}(n)} \vee  5\log (n) \right)^2}{2\left(W_{(i)}(n) + \sqrt{5\log (n) W_{(i)}(n)} \vee  5\log (n) \right) } =& \frac{5^2\log^2 (n)}{2\left(W_{(i)}(n) + 5\log (n)  \right) } \\
\geq& \frac{5^2\log^2 (n)}{20\log (n) } \\
=& \frac{5}{4}\log(n).
\end{align*}
Hence, overall
\begin{align*}
\mathbb{P}\left( \left| D_i(n) - W_{(i)}(n) \right| \geq \sqrt{5\log (n) W_{(i)}(n)} \vee  5\log (n) \mid \{W_{(i)}(n)\}_{i \in [n]} \right) \leq n^{-5/4}.
\end{align*}
\end{proof}

From Lemmas \ref{lem:log} and \ref{lem:Poi_conc} we achieve the following general result regarding concentration for the degrees of the nodes assigned the largest weights. 

\begin{lemma}
\label{lem:d_conc}
Assume that $U$ satisfies \eqref{eq:rv} and \eqref{eq:Ulower}. Suppose $k(n) \log^\alpha(n)/n \rightarrow 0$. Then, under the Norros-Reittu model, as $n \rightarrow \infty$
\begin{align*}
\mathbb{P}&\left( \forall i \in [k(n)],  \left| D_i(n) - W_{(i)}(n) \right| < \sqrt{5\log (n) W_{(i)}(n)}  \right) \rightarrow 1.
\end{align*} 
\end{lemma}

\begin{proof}
Note that from Lemma \ref{lem:log}, $\mathbb{P}\left( W_{(k(n))}(n) > 5\log(n) \right) \rightarrow 1$ so that is suffices to show that
\begin{align*}
\mathbb{P}&\left( \exists i \in [k(n)] :\left| D_i(n) - W_{(i)}(n) \right| \geq \sqrt{5\log (n) W_{(i)}(n)}  , W_{(k(n))}(n) \geq 5\log (n) \right) \rightarrow 0,
\end{align*}
as $n \rightarrow \infty$. Using the union bound, we may upper bound the previous quantity by
\begin{align*}
\sum_{i = 1}^{k(n)} \mathbb{P}&\left( \left| D_i(n) - W_{(i)}(n) \right| \geq \sqrt{5\log (n) W_{(i)}(n)}, W_{(k(n))}(n) \geq 5\log(n) \right) \\
\leq& \sum_{i = 1}^{k(n)} \mathbb{P}\left( \left| D_i(n) - W_{(i)}(n) \right| \geq \sqrt{5\log (n) W_{(i)}(n)}, W_{(i)}(n) \geq 5\log(n) \right), \\
\intertext{condition on $\{W_{(i)}(n)\}_{i \in [n]}$ and apply Lemma \ref{lem:Poi_conc} to achieve that}
=& \mathbb{E}\left[ \mathbb{P}\left(  \left| D_i(n) - W_{(i)}(n) \right| \geq \sqrt{5\log (n) W_{(i)}(n)}\mid \{W_{(i)}(n)\}_{i \in [n]} \right) 1_{\left\lbrace  W_{(i)}(n) \geq 5 \log (n) \right\rbrace}\right] \\
\leq& \mathbb{E}\left[ n^{-5/4} 1_{\left\lbrace W_{(i)}(n) \geq 5 \log (n) \right\rbrace}\right] \\
\leq&  n^{-5/4}. 
\end{align*}
\end{proof}

We now consider the rate at which $D_i(n)$ concentrates around $W_{(i)}(n)$ for $i \in [n]$ in the Chung-Lu model. However, we first require that $\mathbb{E}\left[D_i(n) | \{W_{(i)}(n)\}_{i \in [n]} \right] = W_{(i)}(n)$, $i \in [k(n)]$ with high probability. The following lemma ensures that this is possible when $\alpha > 2$. 
\begin{lemma}
Assume that $U$ satisfies \eqref{eq:rv}. Suppose $\alpha > 2$. As $n \rightarrow \infty$,
\label{lem:maxbound}
\begin{align*}
\mathbb{P}\left( W^2_{(1)}(n) \leq L(n)\right) \rightarrow 1.
\end{align*}
\end{lemma}
\begin{proof}
We follow the proof in \cite{bhattacharjee2022large} Theorem 2.3(D). Let $\delta \in (2/\alpha, 1)$. Then
\begin{align*}
\mathbb{P}\left( W^2_{(1)}(n) > L(n)\right) \leq& \mathbb{P}\left( W^2_{(1)}(n) > n^\delta \right) + \mathbb{P}\left( L(n) < n^\delta \right). 
\end{align*}
Note that by Chebyshev's inequality
\begin{align*}
\mathbb{P}\left( L(n) < n^\delta \right) \leq \frac{n \text{Var}[W_1]}{(n^\delta - n\mathbb{E}[W_1])^2} =& \frac{1}{n}\frac{\text{Var}[W_1]}{(n^{\delta-1} - \mathbb{E}[W_1])^2} \rightarrow 0.
\end{align*}
On the otherhand, we may apply the union bound to achieve that 
\begin{align*}
\mathbb{P}\left( W^2_{(1)}(n) > n^\delta \right)  \leq n\mathbb{P}\left( W^2_{1} > n^\delta \right) \sim n^{1-\frac{\alpha \delta}{2}} \ell(n^{\frac{\delta}{2}}),
\end{align*}
for a slowly varying function $\ell$. Hence $\mathbb{P}\left( W^2_{(1)}(n) > n^\delta \right) \rightarrow 0$ as $n \rightarrow \infty$ and the proof is complete. 
\end{proof}

Our next result again considers the concentration of  $D_i(n)$ around $W_{(i)}(n)$ for $i \in [k(n)]$ using Chernoff's inequality which we present in Lemma \ref{lem:Chernoff}.

\begin{lemma}
\label{lem:CL_conc}
Assume that $U$ satisfies \eqref{eq:rv} and \eqref{eq:Ulower}. Suppose $\alpha > 2$ and $k(n) \log^\alpha(n)/n \rightarrow 0$ as $n \rightarrow \infty$. Then, under the Chung-Lu model, as $n \rightarrow \infty$
\begin{align*}
\mathbb{P}\left( \forall i \in [k(n)],  \left| D_i(n) - W_{(i)}(n) \right| < \sqrt{3\log (n) W_{(i)}(n)}  \right) \rightarrow 1.
\end{align*} 
\end{lemma}
\begin{proof}
By partitioning on the events
\begin{align*}
R(n) = \left\lbrace W^2_{(1)}(n) \leq L(n) \right\rbrace, V(n) = \left\lbrace W_{(k(n))}(n) \geq 3\log(n) \right\rbrace,
\end{align*}
and applying Lemmas \ref{lem:log} and \ref{lem:maxbound} it remains to show that 
\begin{align*}
\mathbb{P}&\left( \exists i \in [k(n)] :  \left| D_i(n) - W_{(i)}(n) \right| \geq \sqrt{3\log (n) W_{(i)}(n)}, R(n), V(n) \right) \rightarrow 0
\end{align*}
Employing the union bound gives that
\begin{align*}
\mathbb{P}&\left( \exists i \in [k(n)] : \left| D_i(n) - W_{(i)}(n) \right| \geq \sqrt{3\log (n) W_{(i)}(n)}, R(n), V(n)  \right) \\
\leq& \sum_{i = 1}^{k(n)} \mathbb{P}\left( \left| D_i(n) - W_{(i)}(n) \right| \geq \sqrt{3\log (n) W_{(i)}(n)}, R(n), V(n)  \right) \\
=& \sum_{i = 1}^{k(n)} \mathbb{E}\left[ \mathbb{P}\left( \left| D_i(n) - W_{(i)}(n) \right| \geq \sqrt{3\log (n) W_{(i)}(n)} \mid \{W_{(i)}(n)\}_{i \in [n]} \right)1_{R(n) \cap V(n)} \right],
\intertext{and under the event $R(n)$, $W_{(i)}(n) = \mathbb{E}\left[D_i(n) | \{W_{(i)}(n)\}_{i \in [n]} \right]$ for all $i \in [k(n)]$. Further, under $V(n)$ $3\log(n) /  W_{(i)}(n) < 1$ so that we may apply Chernoff's bound (see Lemma \ref{lem:Chernoff}) to achieve that}
=& \sum_{i = 1}^{k(n)} \mathbb{E}\left[2 n^{-1 }1_{R(n)\cap V(n)} \right] \\
\leq& 2k(n)/n \rightarrow 0. 
\end{align*}
\end{proof}
In addition, we require a variant of the previous result that holds for all $i \in [n]$. That is, we require high probability upper bounds for all of the degrees generated from the Chung-Lu model.  
\begin{lemma}
\label{lem:d_upper}
Assume that $U$ satisfies \eqref{eq:rv}. Suppose $\alpha > 2$. Under the Chung-Lu model, as $n \rightarrow \infty$
\begin{align*}
\mathbb{P}\left( \forall i \in [n], D_i(n) <  W_{(i)}(n) + \sqrt{4\log(n) W_{(i)}(n)}\vee 4\log(n) \right) \rightarrow 1. 
\end{align*}
\end{lemma}

\begin{proof}
Recall the event
\begin{align*}
R(n) = \left\lbrace W^2_{(1)}(n) \leq L(n) \right\rbrace.
\end{align*}
We partition,
\begin{align*}
\mathbb{P}&\left( \exists i \in [n] : D_i(n) \geq  W_{(i)}(n) + \sqrt{4\log(n) W_{(i)}(n)}\vee 4\log(n) \right) \\
\leq& \mathbb{P}\left( \exists i \in [n] : D_i(n) \geq  W_{(i)}(n) + \sqrt{4\log(n) W_{(i)}(n)}\vee 4\log(n), R(n) \right) + \mathbb{P}\left(R^c(n)\right).
\end{align*}
By Lemma \ref{lem:maxbound}, the second hand term converges to zero as $n \rightarrow \infty$. Hence it suffices to show that the first term is asymptotically negligible. See that
\begin{align*}
\mathbb{P}&\left( \exists i \in [n] : D_i(n) \geq  W_{(i)}(n) + \sqrt{4\log(n) W_{(i)}(n)}\vee 4\log(n), R(n) \right) \\
\leq& \sum_{i = 1}^n \mathbb{P}\left( D_i(n) \geq  W_{(i)}(n) + \sqrt{4\log(n) W_{(i)}(n)}\vee 4\log(n), R(n) \right) \\
=& \sum_{i = 1}^n \mathbb{E}\left[\mathbb{P}\left( D_i(n) \geq  W_{(i)}(n) + \sqrt{4\log(n) W_{(i)}(n)}\vee 4\log(n) \mid \{W_{(i)}(n)\}_{i \in [n]} \right) 1_{R(n)} \right],
\intertext{and on $R(n)$, $\mathbb{E}\left[D_i(n) \mid  \{W_{(i)}(n)\}_{i \in [n]} \right] = W_{(i)}(n)$ for all $i \in [n]$ so that we may apply Lemma \ref{lem:Chernoff} to achieve that}
\leq& \sum_{i = 1}^n \mathbb{E}\left[\exp\left\lbrace -\frac{\left( \sqrt{\frac{4\log(n)}{W_{(i)}(n)}}\vee \frac{4\log(n)}{W_{(i)}(n)} \right)^2}{2 + \sqrt{\frac{4\log(n)}{W_{(i)}(n)}}\vee \frac{4\log(n)}{W_{(i)}(n)}} W_{(i)}(n)  \right\rbrace  1_{R(n)} \right]. 
\end{align*}
Note that if  $W_{(i)}(n) > 4\log(n)$
\begin{align*}
\frac{\left( \sqrt{\frac{4\log(n)}{W_{(i)}(n)}}\vee \frac{4\log(n)}{W_{(i)}(n)} \right)^2}{2 + \sqrt{\frac{4\log(n)}{W_{(i)}(n)}}\vee \frac{4\log(n)}{W_{(i)}(n)}} W_{(i)}(n)   =&  \frac{ \frac{4\log(n)}{W_{(i)}(n)}}{2 + \sqrt{\frac{4\log(n)}{W_{(i)}(n)}}} W_{(i)}(n) \\
>& \frac{4\log(n)}{3W_{(i)}(n)} W_{(i)}(n) \\
=& n^{-4/3},
\end{align*}
and if $W_{(i)}(n) \leq 4\log(n)$
\begin{align*}
\frac{\left( \sqrt{\frac{4\log(n)}{W_{(i)}(n)}}\vee \frac{4\log(n)}{W_{(i)}(n)} \right)^2}{2 + \sqrt{\frac{4\log(n)}{W_{(i)}(n)}}\vee \frac{4\log(n)}{W_{(i)}(n)}} W_{(i)}(n)  =&  \frac{\left( \frac{4\log(n)}{W_{(i)}(n)} \right)^2}{2 + \frac{4\log(n)}{W_{(i)}(n)}} W_{(i)}(n)  \\
=& \frac{4^2\log^2(n) }{2W_{(i)}(n) + 4\log(n)}    \\
\geq&  \frac{4^2\log^2(n) }{12\log(n)} = n^{-4/3}.
\end{align*}
Hence
\begin{align*}
\mathbb{P}\left( \exists i \in [n] : D_i(n) >  W_{(i)}(n) + \sqrt{4\log(n) W_{(i)}(n)}\vee 4\log(n), R(n) \right) \leq& \sum_{i = 1}^n n^{-4/3} \rightarrow 0,
\end{align*}
as $n \rightarrow \infty$. 
\end{proof}

Here we present the Chernoff bounds \cite{hagerup1990guided, vershynin2018high}. 
\begin{lemma}
\label{lem:Chernoff}
For any $i \in [n]$, $\delta \in (0, 1)$ 
\begin{align*}
\mathbb{P}\left( \left| D_i(n) - \mathbb{E}\left[D_i(n) | \{W_{(i)}(n)\}_{i \in [n]} \right] \right| \geq \delta \mathbb{E}\left[D_i(n) | \{W_{(i)}(n)\}_{i \in [n]} \right] \mid   \{W_{(i)}(n)\}_{i \in [n]} \right) \leq 2e^{-\delta^2 \mathbb{E}\left[D_i(n) | \{W_{(i)}(n)\}_{i \in [n]} \right]/ 3}, 
\end{align*}
meanwhile for any $\delta > 0$,
\begin{align*}
\mathbb{P}\left( D_i(n)  \geq \left( 1 + \delta\right) \mathbb{E}\left[D_i(n) | \{W_{(i)}(n)\}_{i \in [n]} \right] \mid   \{W_{(i)}(n)\}_{i \in [n]} \right) \leq e^{-\delta^2 \mathbb{E}\left[D_i(n) | \{W_{(i)}(n)\}_{i \in [n]} \right]/ (2 + \delta)}.
\end{align*}
\end{lemma}

\subsection{Upper order statistic alignment}\label{sec:ord_proof}

We may now combine Lemmas \ref{lem:w_repel}, \ref{lem:d_conc} and \ref{lem:CL_conc} to state a general result about the upper $k(n)$ node degrees drawn from either the Norros-Reittu model or the Chung-Lu model with $\alpha > 2$.

\begin{theorem}
\label{thm:C&S}
Assume that $U$ satisfies \eqref{eq:rv} and \eqref{eq:Ulower}. Assume that the network is generated under either the Norros-Reittu random graph with $\alpha > 0$ or the Chung-Lu random graph with $\alpha > 2$. Suppose $k(n)$ is such that $k(n) \rightarrow \infty$ and $k(n) \log^{\alpha}(n) /n^{\frac{1}{4\alpha + 1}} \rightarrow 0$. Then
\begin{align*}
\mathbb{P}\left(D_1(n) > D_2(n) > \dots > D_{k(n)}(n) \right) \rightarrow 1,
\end{align*}
as $n \rightarrow \infty$. 
\end{theorem}
\begin{proof}
Simply note that
\begin{align*}
\left\lbrace D_1(n) > D_2(n) > \dots > D_{k(n)}(n) \right\rbrace \supset &\left\lbrace \forall i \in [k(n)],  \left| D_i(n) - W_{(i)}(n) \right| < \sqrt{5\log (n) W_{(i)}(n)}  \right\rbrace \cap \\
&\left\lbrace \forall i \in [k(n)],  W_{(i)}(n) - W_{(i + 1)}(n) >  2\sqrt{5\log (n) W_{(i)}(n)} \right\rbrace \\
\equiv& C(n) \cap S(n),
\end{align*}
since on $C(n)$ and $S(n)$ we have that for all $i \in [k(n) - 1]$
\begin{align*}
D_i(n) - D_{i + 1}(n) >& W_{(i)}(n) -  \sqrt{5\log (n) W_{(i)}(n)}  - W_{(i + 1)}(n) -  \sqrt{5\log (n) W_{(i + 1)}(n)} \\
>& W_{(i)}(n)  - W_{(i + 1)}(n) -  2\sqrt{5\log (n) W_{(i)}(n)}  \\
>& 0.
\end{align*}
Hence
\begin{align*}
\mathbb{P}\left(\left\lbrace D_1(n) > D_2(n) > \dots > D_{k(n)}(n) \right\rbrace^c \right) \leq& \mathbb{P}\left(C^c(n) \cup S^c(n)\right) \\
\leq& \mathbb{P}\left(C^c(n)\right)  + \mathbb{P}\left(S^c(n)\right) \\
\rightarrow& 0,
\end{align*}
as $n \rightarrow \infty$, where we have applied Lemma \ref{lem:w_repel} as well as Lemma \ref{lem:d_conc} in the Chung-Lu case and Lemma \ref{lem:CL_conc} in the Norros-Reittu case. 
\end{proof}

The last step towards proving alignment of upper order statistics is to ensure that no degree $D_i(n)$, $i \in [n] \setminus [k(n)]$, exceeds $D_{k(n)}$. This is proved in the following lemma

\begin{lemma}
\label{lem:M}
Assume that $U$ satisfies \eqref{eq:rv} and \eqref{eq:Ulower}. Assume that the network is generated under either the Norros-Reittu model with $\alpha > 0$ or the Chung-Lu model with $\alpha > 2$. Suppose $k(n)$ is such that $k(n) \rightarrow \infty$ and $k(n) \log^{\alpha}(n) /n^{\frac{1}{4\alpha + 1}} \rightarrow 0$ as $n \rightarrow \infty$. Then as $n \rightarrow \infty$
\begin{align*}
\mathbb{P}\left(D_{k(n)}(n) >  \max_{i \in [n] \setminus [k(n)]}D_i(n) \right) \rightarrow 1.
\end{align*}
\end{lemma} 
\begin{proof}
We aim to show that
\begin{align*}
\mathbb{P}\left(D_{k(n)}(n) \leq  \max_{i \in [n] \setminus [k(n)]}D_i(n) \right) \rightarrow 0,
\end{align*}
as $n \rightarrow \infty$. 
Define the events
\begin{align*}
B_1(n) =& \left\lbrace \max_{i \in [n]\setminus[k(n)]} D_i(n) <  W_{k(n) + 1}(n) + \sqrt{5\log(n) W_{k(n) + 1}(n)}\vee 5\log(n) \right\rbrace \\
B_2(n) =& \left\lbrace  D_{k(n)}(n) >  W_{(k(n))}(n) -\sqrt{5\log(n) W_{(k(n))}(n)} \right\rbrace 
\end{align*}
By Lemmas  \ref{lem:Poi_conc} and \ref{lem:d_upper}, we have that
\begin{align*}
\mathbb{P}\left(B_1^c(n) \right) \leq&  \mathbb{P}\left(\exists i \in [n]\setminus[k(n)] : D_i(n) \geq  W_{(k(n) + 1)}(n) + \sqrt{5\log(n) W_{(k(n) + 1)}(n)}\vee 5\log(n)  \right) \\
\leq& \mathbb{P}\left(\exists i \in [n]\setminus[k(n)] : D_i(n) \geq  W_{(i)}(n) + \sqrt{5\log(n) W_{(i)}(n)}\vee 5\log(n)  \right) \\
\rightarrow& 0,
\end{align*}
as $n \rightarrow \infty$. In addition, Lemmas \ref{lem:d_conc} and \ref{lem:CL_conc} give that $\mathbb{P}(B^c_2(n)) \rightarrow 0$ as $n \rightarrow \infty$. Hence, it suffices to show that
\begin{align*}
\mathbb{P}\left(D_{k(n)}(n) \leq  \max_{i \in [n] \setminus [k(n)]}D_i(n), B_1(n), B_2(n) \right) \rightarrow 0,
\end{align*}
as $n \rightarrow \infty$, from which it suffices to show that
\begin{align*}
\mathbb{P}\left(W_{(k(n))}(n) -\sqrt{5\log(n) W_{(k(n))}(n)} \leq   W_{(k(n) + 1)}(n) + \sqrt{5\log(n) W_{(k(n) + 1)}(n)}\vee 5\log(n) \right) \rightarrow 0,
\end{align*}
as $n \rightarrow \infty$. Note that
\begin{align*}
\mathbb{P}&\left(W_{(k(n))}(n) -\sqrt{5\log(n) W_{(k(n))}(n)} \leq   W_{(k(n) + 1)}(n) + \sqrt{5\log(n) W_{(k(n) + 1)}(n)}\vee 5\log(n) \right) \\
\leq&  \mathbb{P}\left(W_{(k(n))}(n) - W_{(k(n) + 1)}(n) \leq   2\sqrt{5\log(n) W_{(k(n))}(n)}\right) + \mathbb{P}\left(W_{(k(n) + 1)}(n) < 5\log(n) \right).
\end{align*}
Both terms in the previous display converge to zero as $n \rightarrow \infty$ by Lemmas \ref{lem:w_repel} and \ref{lem:log} respectively. Hence the proof is complete.
\end{proof}

We are now able to prove the main result, Theorem \ref{thm:order}.

\begin{proof}[Proof of Theorem \ref{thm:order}]
Simply notice that the event
\begin{align*}
\left\lbrace \left(R_1(n), R_2(n), \dots, R_{k(n)}(n)\right) = \left(1, 2, \dots,k(n)\right)  \right\rbrace
\end{align*}
contains the event
\begin{align*}
\left\lbrace D_1(n) > D_2(n) > \dots > D_{k(n)}(n) \right\rbrace \cap \left\lbrace  D_{k(n)}(n) >  \max_{i \in [n] \setminus [k(n)]}D_i(n)  \right\rbrace,
\end{align*}
and apply Theorem \ref{thm:C&S} and Lemma \ref{lem:M}.
\end{proof}

\subsection{Approximating tail index estimators}\label{sec:tail_approx}

In this section we provide proofs showing that the tail index estimators based on the degrees may be approximated by their weight-based counterparts under the alignment of upper order statistics. We first prove the approximation for the Hill and Pickands estimators. 

\begin{lemma}
\label{lem:hill_approx}
Assume that $U$ satisfies \eqref{eq:rv} and \eqref{eq:Ulower}. Suppose the network is generated under either the Norros-Reittu model with $\alpha > 0$ or the Chung-Lu model with $\alpha > 2$. Let $k(n)$ be such that $k(n) \rightarrow \infty$, $k(n) \log^{\alpha}(n) /n^{\frac{1}{4\alpha + 1}} \rightarrow 0$ as $n \rightarrow \infty$. Then for $\ell \in \{\textup{Hill}, \textup{Pick} \}$
\begin{align*}
\sqrt{k(n)}\left|\hat{\gamma}^\ell_D(n) - \hat{\gamma}^\ell_W(n) \right| \xrightarrow{p} 0,
\end{align*}
as $n \rightarrow \infty$. 
\end{lemma}

\begin{proof}
For both tail index estimators, we aim to prove that for any $\epsilon > 0$, as $n \rightarrow \infty$
\begin{align*}
\mathbb{P}\left(\sqrt{k(n)}\left|\hat{\gamma}_D(n) - \hat{\gamma}_W(n) \right| > \epsilon \right) \rightarrow 0.
\end{align*}
It suffices to prove that as $n \rightarrow \infty$
\begin{align*}
\mathbb{P}\left(\sqrt{k(n)}\left|\hat{\gamma}_D(n) - \hat{\gamma}_W(n) \right| > \epsilon, S(n), C(n), M(n)  \right) \rightarrow 0.
\end{align*}
since $P(S^c(n))$, $P(C^c(n))$ and $P(M^c(n))$ all tend towards zero as $n \rightarrow \infty$ by Lemmas \ref{lem:w_repel}, \ref{lem:d_conc}, \ref{lem:CL_conc} and \ref{lem:maxbound}. We begin with the Hill estimator. Recall that on $S(n) \cap C(n) \cap M(n)$, $\left(D_{(1)}(n), \dots, D_{(k(n))}(n)\right) = \left(D_{1}(n), \dots, D_{k(n)}(n)\right)$ so that
\begin{align*}
\sqrt{k(n)}&\left|\hat{\gamma}^\text{Hill}_D(n) - \hat{\gamma}^\text{Hill}_W(n) \right| \\
=& \left| \frac{\sqrt{k(n)}}{k(n) - 1}\sum_{i = 1}^{k(n) - 1}\log\left(D_{i}(n)/W_{(i)}(n) \right) - \sqrt{k(n)}\log\left(D_{k(n)}(n)/W_{(k(n))}(n) \right)  \right| \\
\leq& \frac{1}{\sqrt{k(n) - 1}}\sum_{i = 1}^{k(n) - 1} \left|\log\left(D_{i}(n)/W_{(i)}(n) \right)\right| + \sqrt{k(n)}\left|\log\left(D_{k(n)}(n)/W_{(k(n))}(n) \right)  \right|.
\end{align*}
Note that for $x, y \geq 0$, if $x > y$
\begin{align}
\label{eq:log1}
\left|\log\left( \frac{x}{y} \right) \right| = \log\left( \frac{x}{y} \right) = \log\left( 1 + \frac{x - y}{y} \right) \leq& \frac{x - y}{y} = \frac{|x-y|}{x \wedge y}. 
\end{align}
On the other hand, if $y > x$
\begin{align}
\label{eq:log2}
\left|\log\left( \frac{x}{y} \right) \right| = \left|-\log\left( \frac{y}{x} \right) \right| = \log\left( \frac{y}{x} \right) = \log\left( 1 + \frac{y - x}{x} \right) \leq \frac{y-x}{x} = \frac{|x-y|}{x \wedge y}.
\end{align}
Hence on $S(n) \cap C(n) \cap M(n)$ we may further bound
\begin{align*}
\sqrt{k(n)}&\left|\hat{\gamma}^\text{Hill}_D(n) - \hat{\gamma}^\text{Hill}_W(n) \right| \\ 
\leq& \frac{\sqrt{k(n)}}{k(n) - 1}\sum_{i = 1}^{k(n) - 1} \frac{\left|D_{i}(n) - W_{(i)}(n) \right|}{D_{i}(n) \wedge W_{(i)}(n)} +  \sqrt{k(n)} \frac{\left|D_{k(n)}(n) - W_{(k(n))}(n) \right|}{D_{k(n)}(n) \wedge W_{(k(n))}(n)},
\intertext{and on $C(n)$ specifically, we have that $\left|D_{i}(n) - W_{(i)}(n) \right| \leq \sqrt{5\log(n)  W_{(i)}(n)}$ so that
} 
\leq&  \frac{\sqrt{k(n)}}{k(n) - 1}\sum_{i = 1}^{k(n) - 1} \frac{\sqrt{5\log(n)  W_{(i)}(n)}}{W_{(i)}(n) - \sqrt{5\log(n)  W_{(i)}(n)}} \\
&+  \sqrt{k(n)} \frac{\sqrt{5\log(n)  W_{(k(n))}(n)}}{W_{(k(n))}(n) - \sqrt{5\log(n)  W_{(k(n))}(n)}} \\
=&  \frac{\sqrt{k(n)}}{k(n) - 1}\sum_{i = 1}^{k(n) - 1} \frac{1}{\sqrt{\frac{W_{(i)}(n)}{5\log(n)}} -1} + \sqrt{k(n)} \frac{1}{\sqrt{\frac{W_{(k(n))}(n)}{5\log(n)}} -1} \\
\leq& 2\sqrt{k(n)} \frac{1}{\sqrt{\frac{W_{(k(n))}(n)}{5\log(n)}} -1}.
\end{align*}
In summary, we have achieved that
\begin{align*}
\mathbb{P}\left(\sqrt{k(n)}\left|\hat{\gamma}^\text{Hill}_D(n) - \hat{\gamma}^\text{Hill}_W(n) \right| > \epsilon, S(n), C(n), M(n)  \right) \leq& \mathbb{P}\left(2\sqrt{k(n)} \frac{1}{\sqrt{\frac{W_{(k(n))}(n)}{5\log(n)}} -1}  > \epsilon \right) \\
=& \mathbb{P}\left(\sqrt{\frac{W_{(k(n))}(n)}{5k(n)\log(n)}} -\frac{1}{\sqrt{k(n)}}  < 2\epsilon^{-1} \right),
\end{align*}
Hence it suffices that $W_{(k(n))}(n)/k(n)\log(n) \xrightarrow{p} \infty$. Since $W_{(k(n))}(n)/U(n/k(n)) \xrightarrow{p} 1$ \citep[see Theorem 4.2 of][]{resnick2007heavy}, it suffices that $U(n/k(n))/k(n)\log(n) \rightarrow \infty$. Using \eqref{eq:Ulower}, this is easily seen to hold under growth rate restrictions on $k(n)$. 

We now follow a similar strategy for the Pickands estimator. Refine $S(n)$, $C(n)$ and $M(n)$ by employing the upper $4k(n)$ order statistics instead of $k(n)$. Again, on $S(n) \cap C(n) \cap M(n)$, $\allowbreak \left(D_{(1)}(n), \dots, D_{(4k(n))}(n)\right) = \left(D_{1}(n), \dots, D_{4k(n)}(n)\right)$ so that
\begin{align*}
\sqrt{k(n)}&\left|\hat{\gamma}^\text{Pick}_D(n) - \hat{\gamma}^\text{Pick}_W(n) \right| \\
\leq&  \frac{\sqrt{k(n)}}{\log 2}\left| \log\left(\frac{D_{k(n)}(n) - D_{2k(n)}(n)}{W_{(k(n))}(n) - W_{(2k(n))}(n)}\right)  \right| +  \frac{\sqrt{k(n)}}{\log 2}\left| \log\left(\frac{D_{2k(n)}(n) - D_{4k(n)}(n)}{W_{(2k(n))}(n) - W_{(4k(n))}(n)}\right)  \right|.
\end{align*}
Further, using \eqref{eq:log1} and \eqref{eq:log2}
\begin{align*}
\left| \log\left(\frac{D_{k(n)}(n) - D_{2k(n)}(n)}{W_{(k(n))}(n) - W_{(2k(n))}(n)}\right) \right| \leq& \frac{\left|D_{k(n)}(n) -  W_{(k(n))}(n)\right| + \left| D_{2k(n)}(n) -  W_{(2k(n))}(n) \right|}{\left(D_{k(n)}(n) - D_{2k(n)}(n) \right)\wedge\left(W_{(k(n))}(n) - W_{(2k(n))}(n) \right)}, \\
\intertext{and on $C(n)$ we have that} 
\leq& \frac{2 \sqrt{5\log(n) W_{(k(n))}(n)}}{W_{(k(n))}(n) - W_{(2k(n))}(n) - 2 \sqrt{5\log(n) W_{(k(n))}(n)}}.
\end{align*}
In addition, on $S(n)$
\begin{align*}
W_{(k(n))}(n) - W_{(2k(n))}(n) = \sum_{i = k(n)}^{2k(n) - 1}\left( W_{(i)}(n) - W_{(i + 1)}(n)\right) \geq 2k(n) \sqrt{5 \log(n) W_{(2k(n))}(n)}, 
\end{align*}
so that
\begin{align*}
\left| \log\left(\frac{D_{k(n)}(n) - D_{2k(n)}(n)}{W_{(k(n))}(n) - W_{(2k(n))}(n)}\right) \right| \leq& \frac{\sqrt{W_{(k(n))}(n)/W_{(2k(n))}(n)}}{k(n) -  \sqrt{W_{(k(n))}(n)/W_{(2k(n))}(n)}} \equiv J_1(n).
\end{align*}
In a similar fashion, it can be shown that on $S(n) \cap C(n) \cap M(n)$,
\begin{align*}
\left| \log\left(\frac{D_{2k(n)}(n) - D_{4k(n)}(n)}{W_{(2k(n))}(n) - W_{(4k(n))}(n)}\right) \right| \leq& \frac{\sqrt{W_{(2k(n))}(n)/W_{(4k(n))}(n)}}{2k(n) -  \sqrt{W_{(2k(n))}(n)/W_{(4k(n))}(n)}} \equiv J_2(n).
\end{align*}
Hence, in summary, we have shown that for the Pickands estimator
\begin{align*}
\mathbb{P}&\left(\sqrt{k(n)}\left|\hat{\gamma}^\text{Pick}_D(n) - \hat{\gamma}^\text{Pick}_W(n) \right| > \epsilon, S(n), C(n), M(n)  \right) \leq \mathbb{P}\left(\sqrt{k(n)}(J_1(n) + J_2(n))/\log 2 > \epsilon  \right).
\end{align*}
It is straightforwardly seen that $W_{(k(n))}(n)/W_{(2k(n))}(n) \xrightarrow{p} 2^{\gamma}$ as $n \rightarrow \infty$ so that $\sqrt{k(n)}J_1(n), \sqrt{k(n)}J_2(n) \xrightarrow{p} 0$ and the proof is complete.
\end{proof}

We may now prove the main result regarding the Hill and Pickands estimator, Theorem \ref{thm:hill}. 

\begin{proof}[Proof of Theorem \ref{thm:hill}]
For $\ell \in \{\text{Hill}, \text{Pick} \}$, simply rewrite 
\begin{align*}
\sqrt{k(n)}\left(\hat{\gamma}^\ell_D(n) - \gamma\right) = \sqrt{k(n)}\left(\hat{\gamma}^\ell_D(n) - \hat{\gamma}^\ell_W(n)\right)  + \sqrt{k(n)}\left(\hat{\gamma}^\ell_W(n) - \gamma\right).
\end{align*}
The first converges in probability to zero by Lemma \ref{lem:hill_approx}, while the second is asymptotically normal by Theorems 3.2.5 and 3.3.5 of \cite{haan2006extreme}. Hence, we may apply Slutsky's Theorem to achieve the result. 
\end{proof}

The approximation of the probability-weighted moment estimator requires a slightly different proof strategy. Instead of approximating the degree-based estimator directly, we approximate the probability-weighted moment estimates. We may then apply standard techniques to achieve asymptotic normality of the degree-based PWM estimator.

\begin{lemma}
\label{lem:pwm_approx}
Assume that $U$ satisfies \eqref{eq:rv} and \eqref{eq:Ulower}. Assume that the network is generated under either the Norros-Reittu or Chung-Lu model with $\alpha > 2$. Let $k(n)$ be such that $k(n) \rightarrow \infty$, $k(n) \log^{\alpha}(n) /n^{\frac{1}{4\alpha + 1}} \rightarrow 0$ as $n \rightarrow \infty$. Then for $q \in \{1, 2\}$
\begin{align*}
\sqrt{k(n)}\left| \frac{\hat{I}^{(q)}_D(n)}{a(n/k(n))} - \frac{\hat{I}^{(q)}_W(n)}{a(n/k(n))} \right| \xrightarrow{p} 0,
\end{align*}
as $n \rightarrow \infty$. 
\end{lemma}

\begin{proof}
It suffices to prove that as $n \rightarrow \infty$
\begin{align*}
\mathbb{P}\left(\sqrt{k(n)}\left| \frac{\hat{I}^{(q)}_D(n)}{a(n/k(n))} - \frac{\hat{I}^{(q)}_W(n)}{a(n/k(n))} \right| > \epsilon, S(n), C(n), M(n)  \right) \rightarrow 0.
\end{align*}
since $P(S^c(n))$, $P(C^c(n))$ and $P(M^c(n))$ all tend towards zero as $n \rightarrow \infty$ by Lemmas \ref{lem:w_repel}, \ref{lem:d_conc}, \ref{lem:CL_conc} and \ref{lem:maxbound}. Of course on $S(n) \cap C(n) \cap M(n)$ we have 
\begin{align*}
\hat{I}^{(q)}_D(n) = \frac{1}{k(n) - 1} \sum_{i = 1}^{k(n) - 1}\left(\frac{i}{k(n) - 1}\right)^{q-1}\left(D_{i}(n) - D_{k(n)}(n) \right),
\end{align*}
so that
\begin{align*}
\sqrt{k(n)}\left| \frac{\hat{I}^{(q)}_D(n)}{a(n/k(n))}  - \frac{\hat{I}^{(q)}_D(n)}{a(n/k(n))} \right| \leq& \frac{\sqrt{k(n)}}{a(n/k(n))(k(n) - 1)} \sum_{i = 1}^{k(n) - 1}\left(\frac{i}{k(n) - 1}\right)^{q-1}\left|D_{i}(n) - W_{(i)}(n) \right| \\
&+ \frac{\sqrt{k(n)}}{a(n/k(n))(k(n) - 1)} \left|D_{k(n)}(n) - W_{(k(n))}(n) \right|  \sum_{i = 1}^{k(n) - 1}\left(\frac{i}{k(n) - 1}\right)^{q-1} \\
\leq& \frac{\sqrt{k(n)}}{a(n/k(n))(k(n) - 1)} \sum_{i = 1}^{k(n) - 1}\left|D_{i}(n) - W_{(i)}(n) \right| \\
&+ \frac{\sqrt{k(n)}}{a(n/k(n))} \left|D_{k(n)}(n) - W_{(k(n))}(n) \right|.
\end{align*}
Further, we may employ the concentration bound on $C(n)$ to achieve that
\begin{align*}
\sqrt{k(n)}\left|  \frac{\hat{I}^{(q)}_D(n)}{a(n/k(n))}  - \frac{\hat{I}^{(q)}_D(n)}{a(n/k(n))} \right| \leq& \frac{\sqrt{5k(n)\log(n)}}{a(n/k(n))(k(n) - 1)} \sum_{i = 1}^{k(n) - 1}\sqrt{W_{(i)}(n)}  \\
&+ \frac{\sqrt{5k(n)\log(n) W_{(k(n))}(n)}}{a(n/k(n))} \\
\leq&2 \frac{\sqrt{5k(n)\log(n) W_{(1)}(n)}}{a(n/k(n))}.
\end{align*}
In summary, we have that
\begin{align*}
\mathbb{P}\left(\sqrt{k(n)}\left| \frac{\hat{I}^{(q)}_D(n)}{a(n/k(n))} - \frac{\hat{I}^{(q)}_W(n)}{a(n/k(n))} \right| > \epsilon, S(n), C(n), M(n)  \right) \leq& \mathbb{P}\left( 2 \frac{\sqrt{5k(n)\log(n) W_{(1)}(n)}}{a(n/k(n))} > \epsilon \right).
\end{align*}
Since $\gamma W_{(1)}(n)/a(n) = O_p(1)$ it suffices that $\sqrt{k(n)a(n)\log(n)}/a(n/k(n)) = o(1)$ and using \eqref{eq:Ulower} it further suffices that $\sqrt{k(n)\log(n)a(n)}k^{1/\alpha}(n)/n^{1/\alpha} = o(1)$. This is easily achieved when $\alpha > 2$ and thus the proof is complete.
\end{proof}

We may now prove the main result regarding the PWM estimator, Theorem \ref{thm:pwm_norm}.

\begin{proof}[Proof of Theorem \ref{thm:pwm_norm}]
We may rewrite
\begin{align*}
\sqrt{k(n)}\left( 
\begin{pmatrix} 
\frac{\hat{I}^{(1)}_{W}(n)}{a(n/k(n))} \\ 
\frac{\hat{I}^{(2)}_{W}(n)}{a(n/k(n))} 
\end{pmatrix}
- 
\begin{pmatrix} 
\frac{1}{1 -\gamma} \\ 
\frac{1}{2(2 -\gamma)} 
\end{pmatrix}
\right) =& 
\sqrt{k(n)}\left( 
\begin{pmatrix} 
\frac{\hat{I}^{(1)}_{D}(n)}{a(n/k(n))} \\ 
\frac{\hat{I}^{(2)}_{D}(n)}{a(n/k(n))} 
\end{pmatrix}
- 
\begin{pmatrix} 
\frac{\hat{I}^{(1)}_{W}(n)}{a(n/k(n))} \\ 
\frac{\hat{I}^{(2)}_{W}(n)}{a(n/k(n))}
\end{pmatrix}
\right) 
\\
&+
\sqrt{k(n)}\left( 
\begin{pmatrix} 
\frac{\hat{I}^{(1)}_{W}(n)}{a(n/k(n))} \\ 
\frac{\hat{I}^{(2)}_{W}(n)}{a(n/k(n))} 
\end{pmatrix}
- 
\begin{pmatrix} 
\frac{1}{1 -\gamma} \\ 
\frac{1}{2(2 -\gamma)} 
\end{pmatrix}
\right) .
\end{align*}
Note that under the Euclidean norm, the first term converges in probability to zero as $n \rightarrow \infty$ by Lemma \ref{lem:pwm_approx}. The second term is asymptotically bivariate normal by Proposition 4.1 of \cite{de2024bootstrapping}. Hence, we may apply the converging together theorem to achieve that
\begin{align*}
\sqrt{k(n)}\left( 
\begin{pmatrix} 
\hat{I}^{(1)}_{D}(n)/a(n/k(n)) \\ 
\hat{I}^{(2)}_{D}(n)/a(n/k(n)) 
\end{pmatrix}
- 
\begin{pmatrix} 
\frac{1}{1 -\gamma} \\ 
\frac{1}{2(2 -\gamma)} 
\end{pmatrix}
\right) 
\Rightarrow 
N\left( 
\begin{pmatrix} 
0 \\ 
0 
\end{pmatrix}
,
\begin{pmatrix} 
\sigma_{11} & \sigma_{12} \\
\sigma_{21} & \sigma_{22}
\end{pmatrix}
\right),
\end{align*}
with covariance matrix as determined in \eqref{eq:pwms:norm}. Finally, applying the delta method gives the result.
\end{proof}

\section{Funding}

D. Cirkovic is supported by NSF Grant DMS-2210735. T. Wang is supported by the National Natural Science Foundation of China Grant 12301660 
and the Science and Technology Commission of Shanghai Municipality Grant 23JC1400700. 

\bibliographystyle{plain}
\bibliography{InhomogeneousMRV.bib}

\begin{thebibliography}{10}

\bibitem{aiello2000random}
William Aiello, Fan Chung, and Linyuan Lu.
\newblock A random graph model for massive graphs.
\newblock In {\em Proceedings of the Thirty-second Annual ACM Symposium on
  Theory of Computing}, pages 171--180, 2000.

\bibitem{albert1999diameter}
R{\'e}ka Albert, Hawoong Jeong, and Albert-L{\'a}szl{\'o} Barab{\'a}si.
\newblock Diameter of the world-wide web.
\newblock {\em Nature}, 401(6749):130--131, 1999.

\bibitem{albert2000error}
R{\'e}ka Albert, Hawoong Jeong, and Albert-L{\'a}szl{\'o} Barab{\'a}si.
\newblock Error and attack tolerance of complex networks.
\newblock {\em Nature}, 406(6794):378--382, 2000.

\bibitem{barabasi1999emergence}
Albert-L{\'a}szl{\'o} Barab{\'a}si and R{\'e}ka Albert.
\newblock Emergence of scaling in random networks.
\newblock {\em Science}, 286(5439):509--512, 1999.

\bibitem{beirlant2006statistics}
Jan Beirlant, Yuri Goegebeur, Johan Segers, and Jozef~L Teugels.
\newblock {\em Statistics of extremes: theory and applications}.
\newblock John Wiley \& Sons, 2006.

\bibitem{bhamidi2021multiscale}
Shankar Bhamidi, Souvik Dhara, and Remco van~der Hofstad.
\newblock Multiscale genesis of a tiny giant for percolation on scale-free
  random graphs.
\newblock {\em arXiv preprint arXiv:2107.04103}, 2021.

\bibitem{bhamidi2012novel}
Shankar Bhamidi, Remco van~der Hofstad, and Johan S.~H. van Leeuwaarden.
\newblock {Novel scaling limits for critical inhomogeneous random graphs}.
\newblock {\em The Annals of Probability}, 40(6):2299 -- 2361, 2012.

\bibitem{bhattacharjee2022large}
Chinmoy Bhattacharjee and Matthias Schulte.
\newblock Large degrees in scale-free inhomogeneous random graphs.
\newblock {\em The Annals of Applied Probability}, 32(1):696--720, 2022.

\bibitem{bollobas2007phase}
B{\'e}la Bollob{\'a}s, Svante Janson, and Oliver Riordan.
\newblock The phase transition in inhomogeneous random graphs.
\newblock {\em Random Structures \& Algorithms}, 31(1):3--122, 2007.

\bibitem{boucheron2013concentration}
St{\'e}phane Boucheron, G{\'a}bor Lugosi, and Pascal Massart.
\newblock {\em Concentration Inequalities: A Nonasymptotic Theory of
  Independence}.
\newblock Oxford University Press, 2013.

\bibitem{broido2019scale}
Anna~D Broido and Aaron Clauset.
\newblock Scale-free networks are rare.
\newblock {\em Nature Communications}, 10(1):1017, 2019.

\bibitem{chung2002average}
Fan Chung and Linyuan Lu.
\newblock The average distances in random graphs with given expected degrees.
\newblock {\em Proceedings of the National Academy of Sciences},
  99(25):15879--15882, 2002.

\bibitem{cirkovic2024emergence}
Daniel Cirkovic, Tiandong Wang, and Daren~BH Cline.
\newblock Emergence of multivariate extremes in multilayer inhomogeneous random
  graphs.
\newblock {\em arXiv preprint arXiv:2403.02220}, 2024.

\bibitem{clauset2009power}
Aaron Clauset, Cosma~Rohilla Shalizi, and Mark~EJ Newman.
\newblock Power-law distributions in empirical data.
\newblock {\em SIAM Review}, 51(4):661--703, 2009.

\bibitem{danielsson2001using}
Jon Danielsson, Laurens de~Haan, Liang Peng, and Casper~G de~Vries.
\newblock Using a bootstrap method to choose the sample fraction in tail index
  estimation.
\newblock {\em Journal of Multivariate Analysis}, 76(2):226--248, 2001.

\bibitem{de1996generalized}
Laurens De~Haan and Ulrich Stadtm{\"u}ller.
\newblock Generalized regular variation of second order.
\newblock {\em Journal of the Australian Mathematical Society}, 61(3):381--395,
  1996.

\bibitem{de2024bootstrapping}
Laurens de~Haan and Chen Zhou.
\newblock Bootstrapping extreme value estimators.
\newblock {\em Journal of the American Statistical Association},
  119(545):382--393, 2024.

\bibitem{dietrich2002testing}
Daniel Dietrich, Laurens De~Haan, and Jurg Husler.
\newblock Testing extreme value conditions.
\newblock {\em Extremes}, 5(1):71, 2002.

\bibitem{drees2020minimum}
Holger Drees, Anja Jan{\ss}en, Sidney~I Resnick, and Tiandong Wang.
\newblock On a minimum distance procedure for threshold selection in tail
  analysis.
\newblock {\em SIAM Journal on Mathematics of Data Science}, 2(1):75--102,
  2020.

\bibitem{faloutsos1999power}
Michalis Faloutsos, Petros Faloutsos, and Christos Faloutsos.
\newblock On power-law relationships of the internet topology.
\newblock {\em ACM SIGCOMM Computer Communication Review}, 29(4):251--262,
  1999.

\bibitem{haan2006extreme}
Laurens Haan and Ana Ferreira.
\newblock {\em Extreme value theory: an introduction}, volume~3.
\newblock Springer, 2006.

\bibitem{hagerup1990guided}
Torben Hagerup and Christine R{\"u}b.
\newblock A guided tour of chernoff bounds.
\newblock {\em Information Processing Letters}, 33(6):305--308, 1990.

\bibitem{hall1984best}
Peter Hall and Alan~H Welsh.
\newblock Best attainable rates of convergence for estimates of parameters of
  regular variation.
\newblock {\em The Annals of Statistics}, pages 1079--1084, 1984.

\bibitem{hall1985adaptive}
Peter Hall and Alan~H Welsh.
\newblock Adaptive estimates of parameters of regular variation.
\newblock {\em The Annals of Statistics}, pages 331--341, 1985.

\bibitem{hill1975simple}
Bruce~M Hill.
\newblock A simple general approach to inference about the tail of a
  distribution.
\newblock {\em The Annals of Statistics}, pages 1163--1174, 1975.

\bibitem{holme2019rare}
Petter Holme.
\newblock Rare and everywhere: Perspectives on scale-free networks.
\newblock {\em Nature Communications}, 10(1):1016, 2019.

\bibitem{hosking1987parameter}
Jonathan~RM Hosking and James~R Wallis.
\newblock Parameter and quantile estimation for the generalized pareto
  distribution.
\newblock {\em Technometrics}, 29(3):339--349, 1987.

\bibitem{husler2006testing}
J{\"u}rg H{\"u}sler and Deyuan Li.
\newblock On testing extreme value conditions.
\newblock {\em Extremes}, 9:69--86, 2006.

\bibitem{matsui2013estimation}
Muneya Matsui, Thomas Mikosch, and Laleh Tafakori.
\newblock Estimation of the tail index for lattice-valued sequences.
\newblock {\em Extremes}, 16(4):429--455, 2013.

\bibitem{naulet2021bootstrap}
Zacharie Naulet, Daniel~M Roy, Ekansh Sharma, and Victor Veitch.
\newblock {Bootstrap estimators for the tail-index and for the count statistics
  of graphex processes}.
\newblock {\em Electronic Journal of Statistics}, 15(1):282 -- 325, 2021.

\bibitem{pickands1975statistical}
James Pickands~III.
\newblock Statistical inference using extreme order statistics.
\newblock {\em The Annals of Statistics}, pages 119--131, 1975.

\bibitem{pollard2002user}
David Pollard.
\newblock {\em A user's guide to measure theoretic probability}.
\newblock Number~8. Cambridge University Press, 2002.

\bibitem{price1965networks}
Derek J De~Solla Price.
\newblock Networks of scientific papers: The pattern of bibliographic
  references indicates the nature of the scientific research front.
\newblock {\em Science}, 149(3683):510--515, 1965.

\bibitem{renyi1953theory}
Alfr{\'e}d R{\'e}nyi.
\newblock On the theory of order statistics.
\newblock {\em Acta Math. Acad. Sci. Hung}, 4(2):48--89, 1953.

\bibitem{resnick2007heavy}
Sidney~I Resnick.
\newblock {\em Heavy-tail phenomena: probabilistic and statistical modeling}.
\newblock Springer Science \& Business Media, 2007.

\bibitem{van2017random}
Remco Van Der~Hofstad.
\newblock {\em Random graphs and complex networks}, volume~43.
\newblock Cambridge University Press, 2017.

\bibitem{vershynin2018high}
Roman Vershynin.
\newblock {\em High-dimensional probability: An introduction with applications
  in data science}, volume~47.
\newblock Cambridge University Press, 2018.

\bibitem{voitalov2019scale}
Ivan Voitalov, Pim van~der Hoorn, Remco van~der Hofstad, and Dmitri Krioukov.
\newblock Scale-free networks well done.
\newblock {\em Physical Review Research}, 1(3):033034, 2019.

\bibitem{wang2019consistency}
Tiandong Wang and Sidney~I Resnick.
\newblock Consistency of {H}ill estimators in a linear preferential attachment
  model.
\newblock {\em Extremes}, 22(1):1--28, 2019.

\end{thebibliography}

\end{document}